\documentclass[11pt,onecolumn]{article}

\usepackage[a4paper, margin=0.75in, tmargin=1.25in, bmargin=1.25in]{geometry}
\usepackage{setspace}
\usepackage{arydshln}
\usepackage{amsmath, amssymb, amsthm}
\usepackage{graphicx, graphics, epsfig, color}
\usepackage{adjustbox}
\usepackage[font=footnotesize]{caption}
\usepackage[font=footnotesize]{subcaption}
\usepackage{algorithm, algorithmic}
\usepackage{float}
\usepackage{multirow}
\usepackage{cuted, flushend}
\usepackage{midfloat}
\usepackage{bm}
\usepackage{fancyhdr}
\usepackage{enumitem}
\usepackage[sectionbib]{natbib}
\usepackage{flushend}
\usepackage{hyperref}

\usepackage[title]{appendix}

\usepackage{mathtools}

\newtheorem{theorem}{Theorem}
\newtheorem{proposition}[theorem]{Proposition}
\newtheorem{lemma}[theorem]{Lemma}

\newtheorem{corollary}[theorem]{Corollary}

\newcommand{\toas}{\xrightarrow{\rm{a.s.}}}
\newcommand{\toip}{\xrightarrow{p}}

\newcommand{\hf}{\textstyle{\frac{1}{2}}}

\newcommand{\E}{\mathbb{E}}
\newcommand{\mP}{\mathbb{P}}

\newcommand{\mR}{\mathbb{R}}

\DeclareMathOperator{\tr}{tr}

\DeclareMathOperator{\diag}{diag}

\newcommand{\lhat}{\hat{\ell}}
\newcommand{\lmin}{\lambda_{\rm \min}}
\newcommand{\rhonn}{\rho_{\nu n}}

\newcommand{\ssf}{\mathsf{F}}
\newcommand{\ssm}{\mathsf{m}}

\newcommand{\bx}{\bm{x}}
\newcommand{\by}{\bm{y}}

\def\bb{\boldsymbol}

\def\var{{\rm Var}}

\def\tr{\, {\rm tr}\,}

\def\var1{{\sigma_1^2}}

\def\E{{\mathbb E}}
\def\Pr{{\mathbb P}}

\begin{document}


\title{Asymptotics of eigenstructure of sample correlation matrices for high-dimensional spiked models}

\author{
	David Morales-Jimenez\footnote{ECIT Institute, Queen's University Belfast, UK; email: d.morales@qub.ac.uk}
	\and Iain M. Johnstone\footnote{Department of Statistics, Stanford University, USA; email: imj@stanford.edu, jeha@stanford.edu}
	\and Matthew R. McKay\footnote{ECE Department, Hong Kong University of Science and Technology, Hong Kong; email: m.mckay@ust.hk}
	\and Jeha Yang\footnotemark[2]
}

\date{\today}

\maketitle

\abstract{Sample correlation matrices are widely used, but 
surprisingly little is known about their asymptotic spectral
properties for high-dimensional data beyond the case of
``null models'', for which the data is assumed
to have independent coordinates.
In the class of spiked  models, we apply random
matrix theory to derive asymptotic first-order and distributional
results for both the leading eigenvalues and eigenvectors of sample
correlation matrices, assuming a high dimensional regime in which the
ratio $p/n$, of number of variables $p$ to sample size $n$, converges
to a positive constant.
While the first order spectral properties of sample correlation matrices
match those of sample covariance matrices,
their asymptotic distributions can differ significantly. Indeed, the
correlation-based fluctuations of both sample eigenvalues and
eigenvectors 
are often remarkably smaller than those of their sample covariance
counterparts.



}

\section{Introduction}

Estimation of correlation matrices is among the most fundamental
statistical tasks,  basic to standard methods and widely used
in applications.
Modern examples include viral sequence analysis and
vaccine design in biology \citep{Dahirel11, Quadeer14, Quadeer18},
large portfolio design in finance \citep{Plerou02}, signal detection in
radio astronomy \citep{Lesham01}, and collaborative filtering
\citep{Liu14, Ruan16}, among many others.  In classical statistical
settings, with a limited number of variables $p$ and a
large sample size $n$, the sample correlation matrix
performs well.  Modern applications like those cited
are however often defined by high dimensionality, with
large $p$ and, in many cases,  limited $n$.
In such high-dimensional cases, sample correlation
matrices become inaccurate due to the aggregation of statistical noise
across the matrix coordinates; a fact that is evident from the
eigen-spectrum \citep{ElKaroui2009}. This is particularly important in
the context of Principal Component Analysis (PCA), which often
involves projecting data onto the leading eigenvectors of the sample
correlation matrix; or equivalently, onto those of the sample covariance
matrix after standardisation of the data.

Despite the extensive use of sample correlation matrices, 
there is relatively little theoretical understanding of
their properties, and particularly of their eigen-spectra, in high
dimensions. In contrast, sample covariance matrices
have been studied extensively, and 
a rich  literature now exists (see e.g.,
\cite{yao15}).
Their asymptotic properties have typically been described
in the high-dimensional setting where the number of samples and
variables both grow large, often though not always at the same rate,
based on the theory of random matrices.
Specific first and second order results
for the eigenvalues and eigenvectors
of sample covariance matrices
are reviewed in books by \citet{BaiSilverstein, Couillet11} and \citet{yao15}.


For the spectra of high-dimensional sample \textit{correlation}
matrices, current theoretical results mainly focus on the simplest ``null
model'' scenario, in which the data is assumed independent.
In this null model, 
correlation matrices share many of the same asymptotic properties as
 covariance matrices from independent and identically
distributed (i.i.d.), zero-mean, unit variance data.
Thus, the empirical eigenvalue distribution converges to the
Marchenko-Pastur distribution almost surely \citep{Jiang04}, 
while the largest and
smallest eigenvalues converge to the edges of this distribution
\citep{Jiang04, Xiao2010}.   
Moreover, the rescaled largest and smallest eigenvalues 
asymptotically follow the Tracy-Widom law \citep{bao2012, pillai2012}.
CLTs for linear spectral statistics have also been derived
\citep{Gao14}.
A separate line of work studies the maximum absolute off-diagonal
entry of sample correlation matrices, referred to as ``coherence''
\citep{jiang2004,cai2011,Cai11},  proposed as a
statistic for conducting independence tests, see also
\citet{Cochran95, Mestre17}  and references therein.
\citet{hero2011,hero2012} use a related
statistic to identify variables exhibiting strong
correlations; an approach referred to as ``correlation screening''.

For non-trivial correlation models however, asymptotic  results
about the spectra of sample correlation matrices are quite
scarce. 
Notably, \cite{ElKaroui2009} shows for a fairly general
class of covariance models with bounded spectral norm that, to first order,
the eigenvalues of sample correlation matrices asymptotically coincide
with those of sample covariance matrices with unit-variance data,
generalizing earlier results of \cite{Jiang04} and \cite{Xiao2010}. Under
similar covariance assumptions, recent work has also presented CLTs
for linear spectral statistics of sample correlation matrices
\citep{Mestre17}, extending the work of \cite{Gao14}.
While first order behavior again coincides with that of
sample covariances, the asymptotic
fluctuations however are very different in the sample correlation
case.

\begin{figure*}[h!]
	\centering
	\begin{subfigure}[b]{0.95\textwidth}
		\centering
		\caption{Histogram of the largest sample eigenvalue}
		\includegraphics[width=\textwidth]{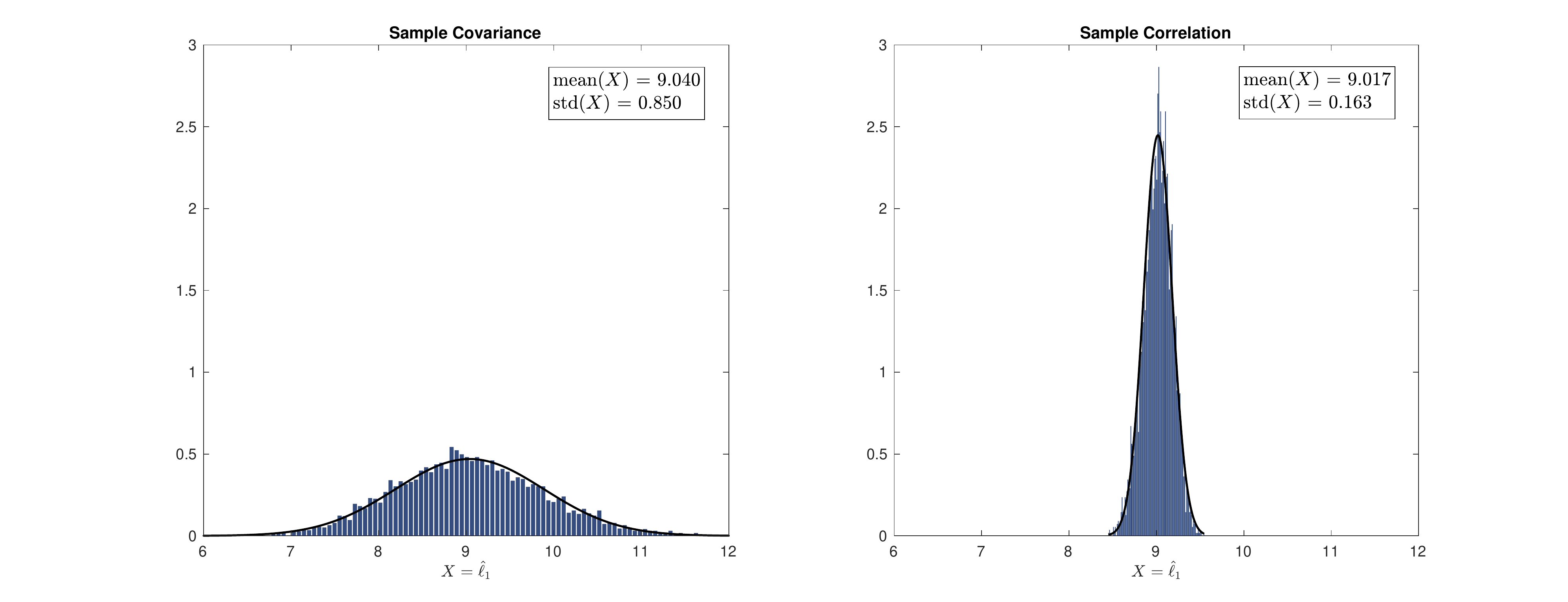}
		\label{fig_motivation_a}
	\end{subfigure}

	\begin{subfigure}[b]{0.95\textwidth}
	\centering
    \caption{Scatter plot of sample-to-population eigenvector projections}
	\includegraphics[width=\textwidth]{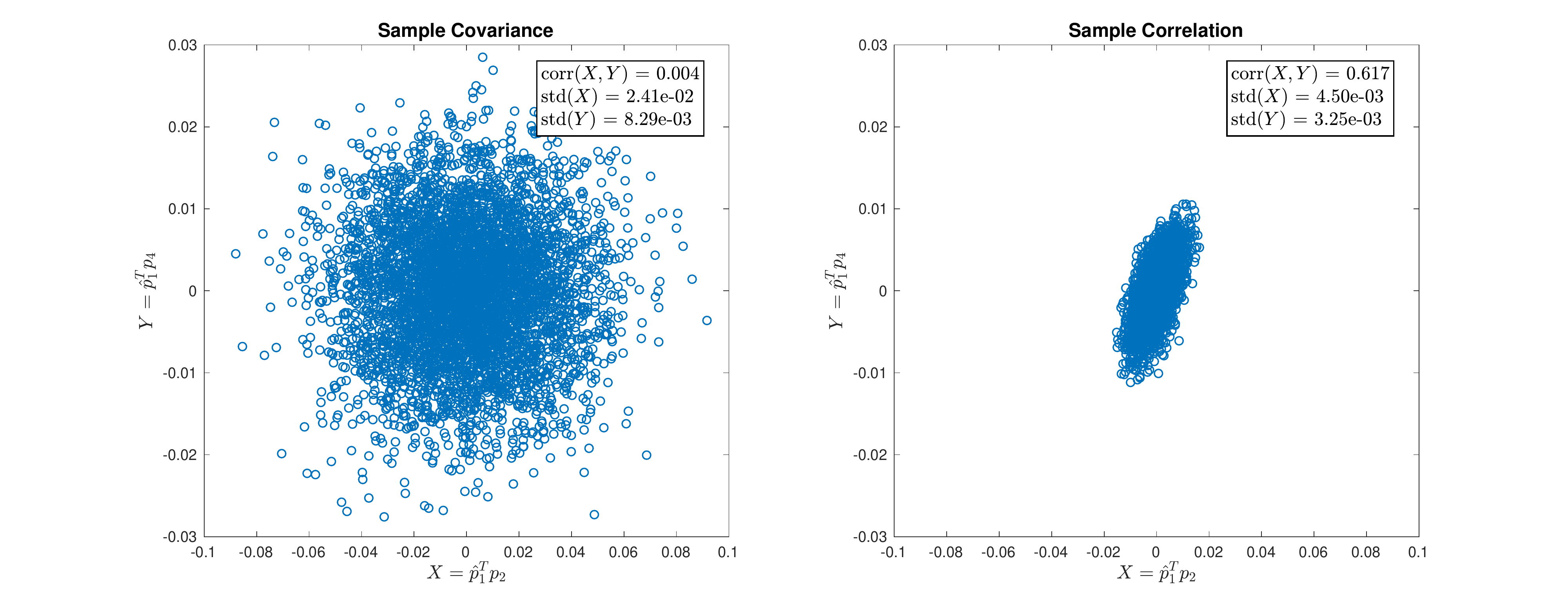}
	\label{fig_motivation_b}
\end{subfigure}
	\caption{A simple simulation example shows remarkable
          distributional differences between the sample covariance and
          sample correlation. From $n=200$ i.i.d. Gaussian samples
          $x_i \in  \mathbb{R}^{100}$ with covariance $\Sigma = {\rm
            blkdiag}(\Sigma_s,I_{90})$, where 
	        $(\Sigma_s)_{i,j=1}^{10} = ( r^{|i-j|} )_{i,j=1}^{10}$ , $r=0.95$, we compute the sample covariance
          and sample correlation and show: (a) the empirical density
          (normalized histogram) of the largest sample eigenvalue,
along with a Gaussian distribution with the estimated mean and
variance (solid line)           
          and (b) scatter plot of the leading sample eigenvector, projected onto the second (x-axis) and fourth (y-axis) population eigenvectors. A striking variance reduction is observed in the sample correlation for both (a) and (b). A similar variance reduction is observed for different choices of the population eigenvectors in (b); the selected choice (second and fourth eigenvectors) facilitates the illustration of an additional correlation effect in the sample-to-population eigenvector projections.}
	\label{fig_motivation}
\end{figure*}


This paper considers a particular class of correlation matrix
models; the so-called ``spiked models'', in which
a few large or small eigenvalues of the population
covariance (or correlation) matrix are assumed to be well separated
from the rest \citep{Johnstone01}.
Spiked covariance models are relevant for applications in
which the primary covariance information lies in a
relatively small number of eigenmodes; for example,  in collaborative
signal detection in cognitive radio systems \citep{Bianchi09},  fault
detection in sensor networks \citep{couillet2013}, adaptive beamforming
in array processing \citep{Hachem13,Vallet15, Yang18}, and protein
contact prediction in biology \citep{Cocco11,Cocco13}.  The spectral
properties of spiked covariance models have now been well studied,
with precise analytical results established for the asymptotic
first-order and distributional properties of both the eigenvalues and
eigenvectors, e.g., \cite{baik2005,baik-silver2006, Paul2007, bai2008,
  Benaych11, couillet2013,Bloemendal2016}.
For reviews, see also \citet[Chapter
9]{Couillet11} and \citet[Chapter 11]{yao15}.

Results for
the spectrum of sample correlation matrices under spiked models are
however far more scarce. While the asymptotic first-order behaviour is
expected to coincide with that of the sample covariance, as a
consequence of \cite{ElKaroui2009}, a simple simulation example
reveals striking differences in the fluctuations of both the sample
eigenvalues and eigenvectors; see Figure \ref{fig_motivation}.

Here, we present theoretical results to describe these observed
phenomena.
We obtain asymptotic first-order and distribution results
for the eigenvalues and eigenvectors of sample correlation matrices
under a spiked model.   In essence, analogues of the set of theorems established by
\cite{Paul2007} for sample covariance matrices in the special case of 
Gaussian data are here presented for sample correlation matrices and generalized to non-Gaussian data.
To first order, the eigenvalues and eigenvectors coincide
asymptotically with those of sample covariance matrices; however,
their fluctuations can be very different.  
Indeed, for both the largest sample correlation eigenvalues, Theorem
\ref{thm_main}, and projections
of the corresponding eigenvectors, Theorem \ref{thm_main_eigenvec}, 
the asymptotic variances 
admit a  decomposition into
three terms.  
The first term is just the asymptotic variance for sample
covariance matrices generated from Gaussian data; the second
adds corrections due to non-Gaussianity, while the
third captures further corrections due to data normalization imposed
by the sample correlation matrix.
(This last amounts to normalizing
the entries of the sample covariance matrix by the sample
variances). 
Consistent with the example of Figure \ref{fig_motivation_a},
in the CLT for the leading sample
eigenvalues, the sample correlation eigenvalues often
show lower fluctuations---despite the variance normalization---than 
those of the sample covariance eigenvalues. 
As seen in Figure \ref{fig_motivation_b}, the (normalized)
eigenvector projections
are typically asymptotically correlated, even for
Gaussian data, unlike the sample covariance setting of 
\citet[Theorem 5]{Paul2007}.



\medskip
{\em Technical contributions:} We build on and extend
a set of random
matrix tools developed in the literature
to study spiked covariance
models. The latter are recalled and organized in a companion manuscript
\citep{Iain}, which gives a parallel treatment for sample covariance
matrices. 
Important adaptations are needed here to account for
the data normalization imposed by sample correlation matrices. Key
technical contributions of our work, basic to our main
theorems, are asymptotic first-order and distributional properties
for bilinear forms and matrix quadratic forms with
normalized entries, Section \ref{sec:normalized}.
A novel regularization-based proof strategy is used to establish the
inconsistency of eigenvector projections in the case of
``subcritical'' spiked eigenvalues, Theorem
\ref{thm_main_eigenvec_subcritical}.

\medskip
\textbf{Model {\rm M}.} \ 
Let $x \in \mathbb{R}^{m+p}$ be a random vector with entries having finite $(4+\delta)$th moment for some $\delta > 0$ and consider the partition
\begin{equation*}
x = \begin{bmatrix}
	\, \xi \, \\
	\eta
	\end{bmatrix} .
\end{equation*}
Assume that $\xi \in  \mathbb{R}^{m}$ has zero mean and covariance $\Sigma$, and is independent of $\eta \in \mathbb{R}^{p}$, which has components $\eta_i$ that are i.i.d. with zero mean and unit variance.
Let $\Sigma_D = {\rm diag}(\sigma_1^2,\ldots,\sigma_m^2)$ be the diagonal matrix containing the variances of $\xi_i$, and let $\Gamma = \Sigma_D^{-1/2} \, \Sigma \, \Sigma_D^{-1/2}$ be the correlation matrix of $\xi$ with eigen-decomposition $\Gamma = P L P^T$, where $P=[p_1,\ldots,p_m]$ is the eigenvector matrix, while $L = {\rm diag}(\ell_1,\ldots,\ell_m)$ contains the spike correlation eigenvalues 
$\ell_1 \geq \ldots \geq \ell_m > 0$.

The correlation matrix of $x$ is therefore $\Gamma_x = {\rm blkdiag}(\Gamma, I)$, with eigenvalues $\ell_1, \ldots, \ell_m, 1, \ldots, 1$ and corresponding eigenvectors ${\mathfrak p}_1, \ldots, {\mathfrak p}_m, e_{m+1}, \ldots, e_{m+p}$, where ${\mathfrak p}_i =  [p_i^T \,\, 0_p^T]^T$ and $e_j$ is the $j$th canonical vector (i.e., a vector of all zeros, but one in the $j$th coordinate).



Consider a sequence of i.i.d. copies of $x$, the first $n$ forming
columns of the
$(m+p) \times n$ data matrix $X = (x_{ij})$.
We will assume that $m$ is fixed, while $p$ and $n$ grow large with 
\begin{equation*}
\gamma_n = p/n \to \gamma > 0 \qquad  {\rm as} \quad p,n \to \infty.
\end{equation*}



\medskip
\textbf{Notation.} \ 
Let $S = n^{-1} XX^T$ be the sample covariance matrix and $S_D={\rm diag}(\hat \sigma_1^2,\ldots,\hat \sigma_{m+p}^2)$ the diagonal matrix containing the sample variances. Denote by $R = S_D^{-1/2} \, S \, S_D^{-1/2}$  the sample correlation matrix, with corresponding $\nu$th sample eigenvalue and eigenvector satisfying
\begin{equation*}
R \, \hat {\mathfrak p}_\nu = \hat \ell_{\nu}  \hat {\mathfrak p}_\nu ,
\end{equation*} 
where, for later use, we partition $\hat {\mathfrak p}_\nu = [\hat p_\nu^T ,  \hat v_\nu^T ]^T$, with $\hat p_\nu$ the subvector of $\hat {\mathfrak p}_\nu$ restricted to the first $m$ coordinates.

For $\ell > 1+\sqrt{\gamma}$, define
\begin{equation*}
\rho(\ell,\gamma) = \ell + \gamma \frac{\ell}{\ell-1} \, , \qquad \dot \rho(\ell,\gamma) = \frac{\partial \rho(\ell,\gamma)}{\partial \ell} =  1 - \frac{\gamma}{(\ell-1)^2} .
\end{equation*}
For an index $\nu$ for which $\ell_{\nu} > 1+\sqrt{\gamma}$ is a simple eigenvalue,  set
\begin{equation} \label{rho_nu_def}
\rho_\nu = \rho(\ell_\nu, \gamma) \, , \qquad \rho_{\nu n} = \rho(\ell_\nu, \gamma_n) \, , \qquad \dot \rho_\nu = \dot \rho(\ell_\nu,\gamma) \, , \qquad \dot \rho_{\nu n} = \dot \rho(\ell_\nu,\gamma_n) .
\end{equation}
We refer to eigenvalues satisfying $\ell_{\nu} > 1+\sqrt{\gamma}$ as ``supercritical'', and those satisfying $\ell_{\nu} \leq 1+\sqrt{\gamma}$ as ``subcritical'', with the quantity $1+\sqrt{\gamma}$ referred to as the ``phase transition''. 

To describe and interpret the variance terms in the limiting
distributions to follow, we make some definitions. Let
$\bar{\xi}_i  = \xi_i/\sigma_i$ and $\kappa_{ij}  = \E \bar{\xi}_i \bar{\xi}_j$
denote the scaled components of $\xi$ and their covariances; of course
$\kappa_{ii}=1$.
The corresponding scaled fourth order cumulants are
\begin{equation} \label{kappacum}
 \kappa_{iji'j'} = \E [ \bar{\xi}_i \bar{\xi}_j \bar{\xi}_{i'}
\bar{\xi}_{j'} ] - \kappa_{ij}\kappa_{i'j'} -
\kappa_{ij'}\kappa_{ji'}- \kappa_{ii'}\kappa_{jj'}.
\end{equation}
When $\xi$ is Gaussian, $\kappa_{i j i' j'} \equiv 0$.

The effect of variance scaling in the correlation matrix will be described by further
quadratic functions of $(\bar{\xi}_i)$ defined by
\begin{align}
\chi_{ij} & = \bar{\xi}_i \bar{\xi}_j,  \qquad \qquad  
\psi_{ij}  = \kappa_{ij} (\bar{\xi}_i^2 +
\bar{\xi}_j^2)/2 \label{chipsidef} \\         
\check{\kappa}_{iji'j'} & = {\rm Cov}(\psi_{ij},\psi_{i'j'}) - {\rm
	Cov}(\psi_{ij},\chi_{i'j'}) - {\rm Cov}(\chi_{ij},\psi_{i'j'}) \; .
\label{kcheckdef}
\end{align}

{\emph{Tensor notation:}} For convenience, it is useful to consider $\kappa_{i j i' j'}$ and $\check{\kappa}_{i j i' j'}$ as entries of 4-dimensional tensor arrays $\kappa$ and $\check{\kappa}$ respectively, and define an additional array ${\cal P}^{\mu \mu' \nu \nu'}$ with entries $p_{\mu,i} p_{\mu',j} p_{\nu,i'} p_{\nu',j'}$.  Also, define ${\cal P}^{\nu}$ as ${\cal P}^{\nu \nu \nu \nu}$.  
Finally, for a second array $A$ of the same dimensions,
\begin{equation*} \label{eq:5}
[\mathcal{P}^\nu,A] = \sum_{i,j,i',j'} P^\nu_{iji'j'} A_{iji'j'}.
\end{equation*}

\section{Main results} \label{sec:MainResults}



Our first main result, proved in Section \ref{arguments}, gives the
asymptotic properties of the largest (spike) eigenvalues of the sample
correlation matrix: 

{\theorem \label{thm_main}
Assume Model {\rm M}, and that $\ell_{\nu} > 1+\sqrt{\gamma}$ is a simple eigenvalue. As $p/n \to \gamma > 0$,
 \begin{align}
&(i) \qquad \,\, \hat \ell_{\nu} \toas \rho_\nu \, , \label{eq:SuperSpkEVal}  \\
&(ii) \qquad \sqrt{n}  ( \hat \ell_{\nu} - \rho_{\nu n} ) \xrightarrow{\cal D} N(0,\tilde \sigma_{\nu}^2) \nonumber ,
 \end{align}
 where
  \begin{equation} \label{sigma_thm}
    \tilde \sigma_{\nu}^2 = 2 \dot \rho_\nu \ell_{\nu}^2
    + \dot \rho_\nu^2 [\mathcal{P}^\nu,  \kappa] 
    + \dot \rho_\nu^2 [\mathcal{P}^\nu, \check{\kappa}].
 \end{equation}
}
Centering at $\rho_{\nu n}$ rather than at $\rho_{\nu}$ is important. If, for example, $\gamma_n = \gamma + an^{-1/2}$, then
\begin{equation*} 
\sqrt{n}  ( \hat \ell_{\nu} - \rho_{\nu} ) \xrightarrow{\cal D} N( a \ell_\nu (\ell_{\nu}-1)^{-1} ,\tilde \sigma_{\nu}^2) ,
\end{equation*}
and we see a limiting shift.
Furthermore, it may also be beneficial to consider $\tilde \sigma_{\nu n}^2$ instead of $\tilde \sigma_{\nu}^2$, obtained by replacing $\dot \rho_{\nu}$ with $\dot \rho_{\nu n}$ in \eqref{sigma_thm}, so that
\begin{equation*} 
	\sqrt{n}  ( \hat \ell_{\nu} - \rho_{\nu n} ) / \tilde \sigma_{\nu n} \xrightarrow{\cal D} N( 0 ,1) .
\end{equation*}

The asymptotic first-order limit in (\emph{i}), which follows  as an easy consequence of \cite{ElKaroui2009}, coincides with that of the $\nu$-th largest eigenvalue of a sample covariance matrix computed from data with population covariance $\Gamma$ \citep{Paul2007}.   This implies that, when constructing $R$, normalizing by the sample variances has no effect on the leading eigenvalues, at least to first order.  

Key differences are seen however when looking at the asymptotic
distribution, given in (\emph{ii}), and in the variance formula
\eqref{sigma_thm} in particular.  This can be readily interpreted. The
first term corresponds to the variance in the Gaussian-covariance case
of \cite{Paul2007}, again for samples with covariance $\Gamma$.
The second provides a correction of that result for
non-Gaussian data, see the companion article \cite{Iain}. 
The third term describes the contribution specific to sample correlation matrices, representing the effect of normalizing the data by the sample variances.  
It is often negative, and is evaluated explicitly for Gaussian data in
Corollary \ref{cor_Gaussian} below, proved in Appendix
\ref{cor_Gaussian_proof}.

{\corollary \label{cor_Gaussian}
For $\xi$ Gaussian, the asymptotic variance in Theorem~\ref{thm_main} simplifies to
\begin{equation*} 
\tilde \sigma_{\nu}^2 =   2  \ell_{\nu}^2 \dot \rho_\nu \left[ 1 - \dot \rho_\nu \left( 2 \ell_\nu  \tr P_{D,\nu}^4   - \tr (P_{D,\nu} \Gamma P_{D,\nu})^2    \right)  \right] ,
\end{equation*}
where $P_{D,\nu} = {\rm diag}(p_{\nu,1},\ldots,p_{\nu,m})$.
}

Thus, computing the sample correlation results in the asymptotic
variance being scaled by \\ $1-\dot \rho_{\nu} \Delta_\nu$
relative to the sample covariance, where
\begin{equation*}  \label{var_Gaussian}
\Delta_\nu = 2 \ell_\nu  \tr P_{D,\nu}^4   - \tr (P_{D,\nu} \Gamma P_{D,\nu})^2 = 2 \ell_{\nu} \sum_i p_{\nu,i}^4 -  \sum_{i, j} (p_{\nu,i} \, \kappa_{i j} \, p_{\nu,j})^2
\end{equation*}
is often positive, implying that spiked eigenvalues of the sample correlation often exhibit a smaller variance than those of the sample covariance. Indeed, such variance reduction occurs iff
\begin{equation}  \label{var_Gaussian_comp}
\sum_{i, j} (p_{\nu,i} \, \kappa_{i j} \, p_{\nu,j})^2 < 2 \ell_{\nu} \sum_i p_{\nu,i}^4 = \sum_{i, j} p_{\nu,i} \, \kappa_{i j} \, p_{\nu,j} (p_{\nu,i}^2 + p_{\nu,j}^2),
 \end{equation}
with the last identity following from the fact that $\ell_{\nu}
p_{\nu,i} = \sum_j \kappa_{ij} \, p_{\nu,j}$.
Condition \eqref{var_Gaussian_comp}, and variance reduction, holds in the
following cases:

\medskip
\begin{tabular}[h]{ll}
  \qquad \qquad (i) & both $\Gamma$ and $p_\nu$ have non-negative entries, or \\
  \qquad \qquad (ii) & $2 \ell_\nu \sum_i p_{\nu,i}^4 > 1$, or \\
  \qquad \qquad (iii) & $ 2 \ell_\nu > \ell_1^2$.
\end{tabular}
\medskip

In case (i),  the inequalities $0 \leq p_{\nu,i} \kappa_{i
  j}  p_{\nu,j} \leq 2 p_{\nu,i} p_{\nu,j} \leq  p_{\nu,i}^2 +
p_{\nu,j}^2$ yield \eqref{var_Gaussian_comp}. 
Note that if $\Gamma$ has non-negative entries, then the Perron-Frobenius
theorem establishes the existence of an eigenvector with non-negative
components for $\ell_1$ ; 
further, if $\Gamma$ has positive entries, by the same theorem,
$\ell_1$ is simple associated with an eigenvector with positive components. Case (ii) follows from 
$\sum_{i, j} (p_{\nu,i} \, \kappa_{i j} \, p_{\nu,j})^2 \leq \sum_{i, j} (p_{\nu,i} \, p_{\nu,j})^2 = 1,$ 
and holds if $\ell_{\nu} > m/2$ since $ \sum_i
p_{\nu,i}^4 \geq 1/m$.
Case (iii) follows from the inequalities
$2 p_{\nu,i}^2 p_{\nu,j}^2 \leq p_{\nu,i}^4 + p_{\nu,j}^4$ and
$\sum_j \kappa_{ij}^2 = (\Gamma^2)_{ii} \leq \|\Gamma^2\| = \ell_1^2$.
Note that this is rather special, in that it has nothing to do with eigenvectors, and a necessary condition for
it to hold is $\ell_1 \leq 2$.


Condition \eqref{var_Gaussian_comp} can fail, however. 
For example, for even $m$ and $r\in(0,1)$, 
consider 
\begin{equation*}
  \Gamma =
  \begin{pmatrix}
    1 & -r \\ -r & 1
  \end{pmatrix}
  \otimes
  1_{m/2} 1_{m/2}^T,
\end{equation*}
where $1_{m/2}$ is the $(m/2)$-dimensional vector of all ones, which corresponds to two negatively correlated groups of identical random vectors.
This has simple supercritical eigenvalues $\ell_1 = (1+r)m/2$ and $\ell_2 = (1-r)m/2$
when $m > 2 (1 + \sqrt{\gamma}) / (1-r)$, 
with $p_{\nu,i}^2 = m^{-1}$ for $\nu = 1, 2$.
One finds that $\Delta_2 = (1-2r-r^2)/2 < 0$ for $r > \sqrt{2} - 1$, 
although 
 $\Delta_1> 0$  since $\ell_1 > m/2$ which implies case (ii).




\bigskip

Turn now to the eigenvectors.  Again, fix an index $\nu$ for which
$\ell_{\nu} > 1 + \sqrt{\gamma}$ is a simple eigenvalue of $\Gamma$,
with corresponding eigenvector ${\mathfrak p}_\nu = [p_\nu^T \,\,
0_p^T]^T$.
Recall that $\hat {\mathfrak p}_\nu = [\hat p_\nu^T \,\, \hat
v_\nu^T]^T$ is the $\nu$th sample eigenvector of $R$ and let $a_\nu = \hat
p_\nu / \| \hat p_\nu \|$  be the corresponding normalized subvector of
$\hat {\mathfrak p}_\nu$ restricted to the first $m$ coordinates. The
next result establishes a limit for the eigenvector projection
$\langle \hat {\mathfrak p}_\nu , {\mathfrak p}_\nu \rangle $, and a
CLT for the normalized cross projections $P^T a_\nu = [ p_1^T a_\nu ,
\ldots ,  p_m^T a_\nu ]^T$; see Sections  \ref{sec:EvecInconsistency}
and \ref{sec:EvecFluc}. 


{\theorem \label{thm_main_eigenvec}
	Assume Model {\rm M}, and that $\ell_{\nu} > 1+\sqrt{\gamma}$ is a simple eigenvalue. Then, as $p/n \to \gamma > 0$,
	\begin{align*}
	&(i) \qquad \,\,   \langle \hat {\mathfrak p}_\nu , {\mathfrak p}_\nu \rangle ^2  \toas \dot \rho_\nu \ell_{\nu} / \rho_{\nu}    \, ,  \\
	&(ii) \qquad \sqrt{n}  ( P^T a_\nu  - e_{\nu} ) \xrightarrow{\cal D} N(0,\Sigma_\nu)  ,
	\end{align*}
	where $\Sigma_\nu = {\cal D}_\nu \tilde \Sigma_\nu  {\cal D}_\nu$ with
    \begin{align}
    {\cal D}_\nu &=  \sum_{k \neq \nu}^{m} (\ell_{\nu}-\ell_k)^{-1}
                   e_k e_k^T   \label{Dnudef} \\
      \tilde \Sigma_{\nu, k l }
                 &=   \dot \rho_{\nu}^{-1}  \ell_k \ell_{\nu} \,
                   \delta_{k,l} +  [\mathcal{P}^{k \nu l \nu},
                   {\kappa} ] + [\mathcal{P}^{k \nu l \nu},
                   \check{\kappa}], \label{Sigma-til}
	\end{align}
where $\delta_{k, l} = 1$ if $k= l$ and $0$ otherwise.
}

The CLT result in $(ii)$ can be rephrased in terms of the entries of $a_\nu$, for which we readily obtain $\sqrt{n}  ( a_\nu  - p_{\nu} ) \xrightarrow{\cal D} N(0,P \Sigma_\nu P^T)$; note that $\Sigma_{\nu}$ has zeros in the $\nu$th row and the $\nu$th column.

As for the eigenvalues, Theorem \ref{thm_main_eigenvec} shows that the
spiked eigenvectors of sample correlation matrices exhibit the same
first-order behavior as those of the sample covariance
\citep{Paul2007}. The difference again lies in the asymptotic
fluctuations, captured by the covariance matrix $\Sigma_\nu$.  Note
that this is decomposed as a product of ${\cal D}_\nu$---a diagonal
matrix---and the matrix $\tilde \Sigma_\nu$, which involves the three
terms in \eqref{Sigma-til}.  These terms have similar interpretations
as those discussed previously in \eqref{sigma_thm}.  That is, the
first term captures the asymptotic fluctuations for a
Gaussian-covariance model \citep{Paul2007}, the second term captures
the effect of non-Gaussianity in the covariance case \citep{Iain}, and
the third term captures information specific to the correlation case,
representing fluctuations due to sample variance normalization.   Note
that only the first term is diagonal in general, suggesting that the
eigenvector projections may be asymptotically correlated,
as seen earlier in Figure \ref{fig_motivation_b}, right panel.
This holds
also for Gaussian data, evaluated explicitly in Corollary
\ref{cor_eigenvec_Gaussian} below, see Appendix
\ref{cor_eigenvec_Gaussian_proof}.
We note an interesting contrast with 
the eigenvector projections for covariance matrices
\citep{Paul2007}, described only by the leading term in \eqref{Sigma-til}.   


{\corollary \label{cor_eigenvec_Gaussian}
For $\xi$ Gaussian, the asymptotic covariance in Theorem~\ref{thm_main_eigenvec} reduces to $\Sigma_\nu = {\cal D}_\nu \tilde \Sigma_{\nu}   {\cal D}_\nu$,
	\begin{equation*}  \label{covar_Gaussian}
	\tilde \Sigma_\nu =  \frac{\ell_{\nu}}{\dot \rho_\nu} L + (\ell_{\nu}I +L) \, \left( \tfrac12 {\cal Z} - \ell_{\nu} {\cal Y} \right) \, (\ell_{\nu}I +L) + \ell_{\nu} (\ell_{\nu}^2 {\cal Y} - L {\cal Y} L)
	\end{equation*}
	with ${\cal Z} =  P^T P_{D,\nu} (\Gamma \circ  \Gamma) P_{D,\nu} P$ and ${\cal Y} = P^T P_{D,\nu}^2 P$, and $\circ$ denoting Hadamard product. 
}

Thus, for Gaussian data, the entries of the asymptotic covariance matrix have (for $k, l \neq \nu$) 
\begin{align*} 
 \Sigma_{\nu, k l}  &= (\ell_{\nu}-\ell_k)^{-1} (\ell_{\nu}-\ell_l)^{-1}  
\left[ \frac{\ell_{\nu}}{\dot \rho_\nu} \ell_k \delta_{k, l}  +  (\ell_{\nu}+\ell_k) (\ell_{\nu}+\ell_l) \frac{{\cal Z}_{k l}}{2} - \ell_\nu \left( \ell_{\nu} (\ell_k + \ell_l) + 2 \ell_k \ell_l  \right) {\cal Y}_{k l} \right]  .
\end{align*}

\bigskip

Consider now the subcritical case in which $\nu$ is such that $1 <
\ell_{\nu} \leq 1+\sqrt{\gamma}$. Denote by ${\mathfrak p}_\nu$ the
corresponding population eigenvector, and $\hat{\ell}_\nu$ and $\hat
{\mathfrak p}_\nu$ the corresponding sample eigenvalue and eigenvector
respectively.  With proofs deferred to Sections \ref{sec:Prelim} and
\ref{sec_eigenvector_subcritical}, we have the following result:  

{\theorem \label{thm_main_eigenvec_subcritical}
	Assume Model {\rm M} and $1 < \ell_{\nu} \leq 1+\sqrt{\gamma}$ is a simple eigenvalue. Then, as $p/n \to \gamma > 0$,
	\begin{align*}
	&(i) \qquad \,\, \hat \ell_{\nu} \toas (1 + \sqrt{\gamma})^2 \, ,  \\
	&(ii) \qquad \,\, 	 \langle \hat {\mathfrak p}_\nu , {\mathfrak p}_\nu \rangle ^2  \toas 0 .
	\end{align*}
}
Once again, the asymptotic first-order limits of the sample eigenvalue
and its associated eigenvector are the same as those obtained for the
sample covariance \citep {Paul2007}. 

To summarize, for the spiked eigenvalues and eigenvectors in both
supercritical and subcritical cases, our results confirm that the
first-order asymptotic behaviour is indeed equivalent to that of
sample covariance matrices, in  agreement with previous results and
observations \citep{ElKaroui2009, Mestre17}.  While the eigenvalue
limits in Theorem \ref{thm_main} and Theorem
\ref{thm_main_eigenvec_subcritical} follow as a straightforward
consequence of \cite{ElKaroui2009}, the eigenvector results of Theorem
\ref{thm_main_eigenvec}-($i$) and Theorem
\ref{thm_main_eigenvec_subcritical}-($ii$) do not.
In contrast to the first-order
equivalences, important differences arise in the fluctuations of both
the eigenvalues and eigenvectors, as shown by the asymptotic
distributions of Theorem \ref{thm_main}-$(ii)$ and Theorem
\ref{thm_main_eigenvec}-$(ii)$.

We illustrate these
differences with a simple example having
covariance $\Gamma = (1-r) I_m + r 1_m 1_m^T$, where $r \in [0, 1]$
; i.e., a model with unit variances
and constant correlation $r$ across all components.  Moreover, $\xi$
is assumed to be Gaussian for simplicity. In this particular setting,
$L = \diag(\ell_1, 1-r, \ldots, 1-r)$, where $\ell_1 = 1 + r(m-1)$ is
supercritical iff $r  > \sqrt \gamma / (m-1)$.
Consider the largest sample eigenvalue $\hat \ell_1$ in such a
supercritical case.
From Corollary \ref{cor_Gaussian}, the asymptotic variances for the
sample covariance and the sample correlation can be respectively computed, yielding
\begin{equation*}
\sigma_{1}^2 = 2 \ell_1^2 \dot \rho_1, \qquad
\tilde{\sigma}_1^2 = \sigma_1^2 (1 - \dot \rho_1 \Delta), \qquad
\end{equation*}
with $\Delta = 2 \ell_1 \tr P_D^4 - \tr (P_D \Gamma P_D)^2$, and where
\begin{equation*}
P_D \triangleq P_{D,1} = m^{-1/2} I_{m}, \qquad
\dot \rho_1 = 1 - \frac{\gamma}{r^2 (m-1)^2} .
\end{equation*}
Figure \ref{fig_single_sector_a} plots these asymptotic variances
versus $r$ for various $(\gamma, m)$.
Indeed the variance (fluctuations) in the sample
correlation are consistently smaller than for the covariance
counterpart. The difference is striking, 
becoming extremely large as $r \nearrow 1$. Similar trends are
observed for various choices of $m$ and $\gamma$,
being more pronounced for higher $m$, while not
much affected by varying $\gamma$.
This may be understood from, after writing
$\Delta = r(2-r) + (1-r)^2 m^{-1}  = 1- (1-r)^2 ( 1 - m^{-1})$,
\begin{equation*}
\frac{\tilde{\sigma}_1^2}{\sigma_1^2}
= 1 - \dot \rho_1 \Delta  
\to
\begin{cases}
\dfrac{\gamma}{(m-1)^2} &   \quad \text{as} \quad r \to 1, \ m
\text{ fixed} \\
(1-r)^2 &   \quad \text{as} \quad m \to \infty, \ r
\text{ fixed}.
\end{cases}
\end{equation*}

\begin{figure}[t] 
	\centering
	\begin{subfigure}[b]{0.9\textwidth}
		\centering
		\caption{Largest sample eigenvalue $\hat \ell_1$}
		\includegraphics[width=\textwidth]{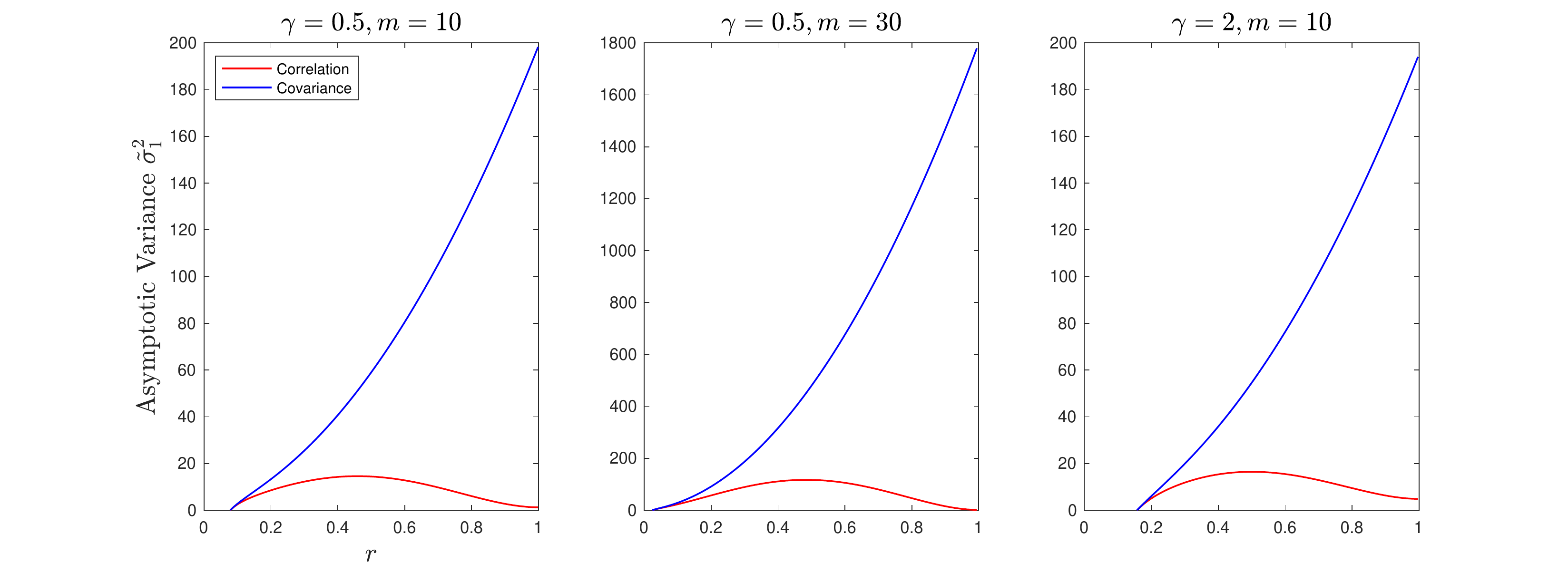}
		\label{fig_single_sector_a}
	\end{subfigure}
	
	\begin{subfigure}[b]{0.9\textwidth}
		\centering
		\caption{Sample-to-population eigenvector projection $p_2^T a_1$}
		\includegraphics[width=\textwidth]{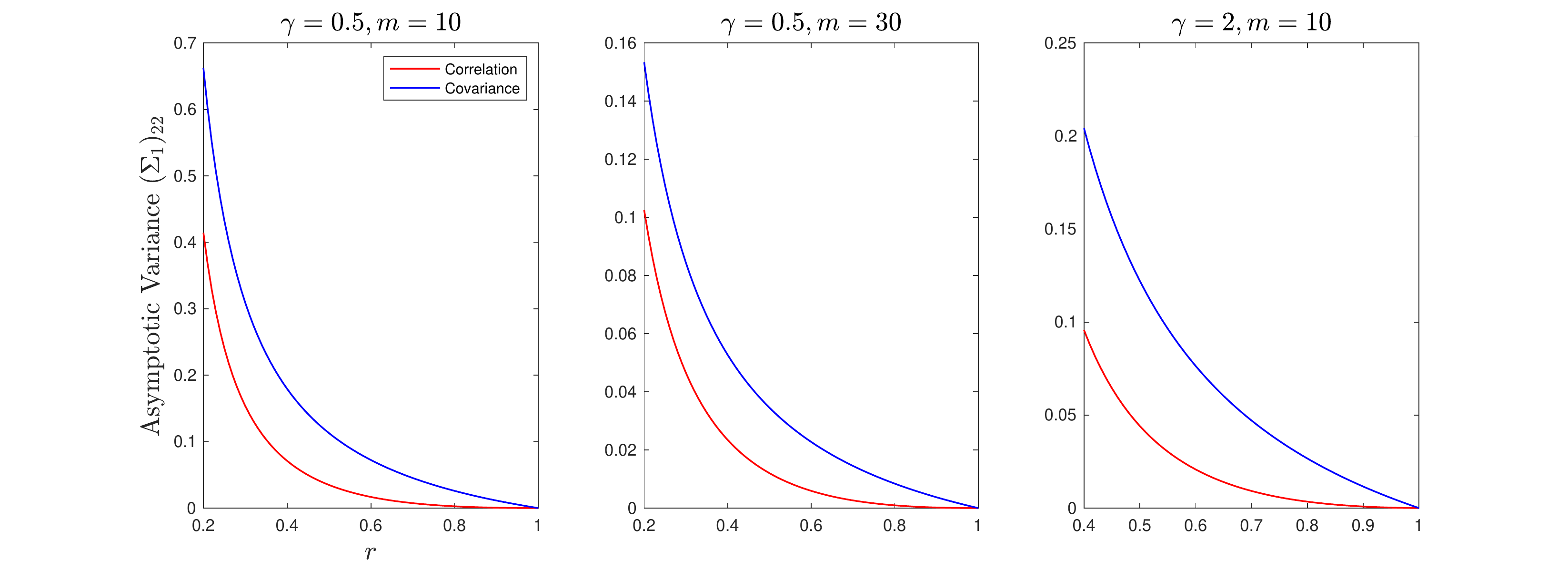}
		\label{fig_single_sector_b}
	\end{subfigure}
	\caption{Differences in the fluctuations of sample eigenvalues and eigenvectors for an example Gaussian model with $\Gamma = (1-r) I_m + r 1_m 1_m^T$. Asympotic variances are shown for (a) the largest sample eigenvalue $\hat \ell_1$ and (b) the normalized sample-to-population eigenvector projection $p_2^T a_1$.}
	\label{fig_single_sector}
\end{figure}


Turn now to the fluctuations of the leading sample eigenvector,
in the same setting as above.
Note that, in Corollary \ref{cor_eigenvec_Gaussian}, for this particular case, one can deduce from $P^T \Gamma P = L$ that
\begin{equation*}
{\cal Z} = m^{-1}  (1-r^2) I_m + r^2 e_1 e_1^T , \qquad
{\cal Y} = m^{-1} I_m.
\end{equation*}
Also from Corollary \ref{cor_eigenvec_Gaussian}, the asymptotic
variances for the normalized sample-to-population eigenvector
projection $p_2^T a_1$, in sample covariance and sample
correlation cases, are respectively computed as 
\begin{equation*}
  \Sigma^{\rm cov}_{1,22} =  \frac{\ell_1 \ell_2}{(r m)^{2} \dot \rho_1}
      , \qquad
      \Sigma_{1,22} = \Sigma^{\rm cov}_{1,22} -  \frac{\zeta}{(r m)^{2}}
                       \frac{\ell_1 \ell_2 (\ell_1+\ell_2)}{m},
\end{equation*}
where $\zeta = 1 - r + \frac12 (1+r) (1+\frac{1-r}{r m})^{-1}$, 
 and we
recall that $\ell_1 = 1 -r + rm$ and $\ell_2 = 1-r$.
These variances are numerically evaluated in Figure
\ref{fig_single_sector_b} for the same parameter choices as before
and, again, as functions of $r$; note however that, for a better
visual appreciation, the range of $r$ has been restricted to
supercritical values sufficiently above the critical point $\sqrt
\gamma / (m-1)$, since the variance explodes at that point. The
comparative evaluation shows once again smaller variances for the
sample correlation.
The variance reduction here is less
visible in the graphs since both $\Sigma_{1, 22}$ and $\Sigma_{1, 22}^{\rm cov}$
vanish as $r \to 1$. 
The ratio, however, behaves quite similarly to the variance ratio
$\tilde{\sigma}_1^2/\sigma_1^2$:
\begin{equation*}
  \frac{\Sigma_{1,22} }{\Sigma^{\rm cov}_{1,22} }
  = 1- \zeta \dot \rho_1 \frac{(\ell_1 + \ell_2)}{m}
  \to
\begin{cases}
\dfrac{\gamma}{(m-1)^2} &   \quad \text{as} \quad r \to 1, \ m
\text{ fixed} \\
(1-r)(1-r/2) &   \quad \text{as} \quad m \to \infty, \ r
\text{ fixed}.
\end{cases}
\end{equation*}


\medskip
\textit{Outline for remainder of paper.} \
First, in Section \ref{sec:preliminaries} we introduce some key quantities and identities that are used in the derivations. Section \ref{sec:normalized} presents necessary asymptotic properties for bilinear forms and matrix quadratic forms with normalized entries, with the corresponding proofs relegated to Section \ref{sec_foundations}. These properties are foundational for describing the asymptotic convergence and distribution of eigenvalues and eigenvectors of sample correlation matrices, derived in Sections \ref{arguments} and \ref{sec_eigenvector} respectively.   

As already noted, a parallel treatment 
for the simpler case of covariance matrices is given in a supplementary
manuscript \citep{Iain}, hereafter abbreviated to [JY].
This aims at a unified exposition of known spectral
properties of spiked covariance matrices as a benchmark for the current
work, along with additional citations to the literature.

\section{Preliminaries}
\label{sec:preliminaries}


We begin with a block representation and some associated reductions for the
sample correlation matrix $R$. These are well known in the covariance
matrix setting. 
As with the
partition of $x$ in Model M, consider 
\begin{equation*}
X = \begin{bmatrix}
	X_1 \\
	X_2
	\end{bmatrix} , \qquad X_1 \in \mathbb{R}^{m \times n}, \quad X_2 \in \mathbb{R}^{p \times n} .
\end{equation*}
Write $S_D = {\rm blkdiag}(S_{D1}, S_{D2})$, with $S_{D1}$
containing the sample variances corresponding to $\xi$ and $S_{D2}$
the ones corresponding to $\eta$. 
Define the `normalized' data matrices $\bar X_1 = S_{D1}^{-1/2} X_1$ and $\bar X_2 = S_{D2}^{-1/2} X_2$, so that
\begin{align*} 
	R =  n^{-1} \begin{bmatrix}
	\bar X_1 \bar X_1^T  &  \bar X_1 \bar X_2^T    \\
	\bar X_2 \bar X_1^T  & \bar X_2 \bar X_2^T 
	\end{bmatrix} 	
	= \begin{bmatrix}
		R_{11} & R_{12}  \\
		R_{21} & R_{22}
	\end{bmatrix}; \qquad
	\hat {\mathfrak p}_\nu = \begin{bmatrix}
		\hat p_\nu \\
		\hat v_\nu
	\end{bmatrix} . 
\end{align*}
This partitioning of the eigenvector equation
$R \, \hat {\mathfrak p}_\nu = \hat \ell_{\nu}  \hat {\mathfrak
  p}_\nu$,
along with $\hat {\mathfrak p}_\nu = [\hat p_\nu^T ,  \hat v_\nu^T ]^T$,
yields
\begin{align*} 
	R_{11} \hat p_\nu + R_{12} \hat v_\nu &= \hat \ell_{\nu} \hat p_\nu \\
	R_{21} \hat p_\nu + R_{22} \hat v_\nu &= \hat \ell_{\nu} \hat v_\nu .
\end{align*}
From the second equation, $\hat v_\nu = (\hat \ell_{\nu} I_p -
R_{22})^{-1} R_{21} \hat p_\nu$.
Putting this into the first equation yields
\begin{align*} 
K(\hat \ell_{\nu}) \hat p_\nu = \hat \ell_{\nu} \hat p_\nu \, , \quad
  {\rm{with}} \quad  K( t ) = R_{11} + R_{12} ( t I_p - R_{22})^{-1}
  R_{21}. 
\end{align*}
Thus $\hat \ell_{\nu}$ is an eigenvalue of $K(\hat \ell_{\nu})$, with
associated eigenvector $\hat p_\nu$; this is central to our
derivations.
Note that $K(\hat \ell_{\nu})$ is well defined if $\hat \ell_{\nu}$ is
well separated from the eigenvalues of $R_{22}$, and 
Section \ref{sec:Prelim} shows that this occurs with probability one
for all large $n$ when $\ell_{\nu}$ is supercritical. Furthermore,
the normalization condition, $\hat p_\nu^T \hat p_\nu + \hat
v_\nu^T \hat v_\nu = 1$ yields 
\begin{equation*} 
	\hat p_\nu^T (I_m + Q_\nu) \hat p_\nu = 1, \qquad Q_\nu = R_{12} (\hat \ell_{\nu} I_p - R_{22})^{-2} R_{21} .
\end{equation*}
Phrased in terms of the signal-space normalized eigenvector $a_\nu = \hat p_\nu / \| \hat p_\nu \|$, we have 
\begin{equation} 
	K(\hat \ell_{\nu}) a_\nu = \hat \ell_{\nu} a_\nu, \qquad \quad 
	a_\nu^T (I_m + Q_\nu) a_\nu =  \| \hat p_\nu \|^{-2}  . \label{av_quad_form}
\end{equation}
Note also that the sample-to-population inner product can be rewritten as
\begin{equation} \label{eig_proj_factored}
	\langle  \hat {\mathfrak p}_\nu,  {\mathfrak p}_\nu  \rangle = \langle  \hat p_\nu,  p_\nu  \rangle =  \| \hat p_\nu \| \langle  a_\nu,  p_\nu  \rangle .
\end{equation}

In the derivation of our CLT results, we use an eigenvector
perturbation formula with quadratic error bound 
given in [JY, Lemma 13],
itself a modification of arguments in
\citet{Paul2007}.
This yields the key expansion    
\begin{equation} \label{eq:anu-decomp}   
a_\nu - p_\nu = - {\cal R}_{\nu n} D_\nu p_\nu + r_\nu,
\end{equation}
where
\begin{equation*}
{\cal R}_{\nu n} = \frac{\ell_{\nu}}{\rho_{\nu n}} \sum_{k \neq \nu}^{m} (\ell_k - \ell_{\nu})^{-1} p_k p_k^T , \qquad
D_\nu = K(\hat \ell_{\nu}) - (\rho_{\nu n}/\ell_{\nu}) \Gamma, \qquad
\| r_\nu \| = O(\| D_\nu \|^2).
\end{equation*}

The derivations of our eigenvalue and eigenvector results,
Sections \ref{arguments} and \ref{arguments_eigenvec} respectively,
take \eqref{av_quad_form}, \eqref{eig_proj_factored} and
\eqref{eq:anu-decomp} as points of departure, and rely on asymptotic
properties of the key objects $K(\hat \ell_{\nu})$ and $Q_\nu$.
In particular, $K( t )$ can be expressed as the random matrix quadratic form
\begin{equation} \label{Kt_def}
K(t) = n^{-1} \bar X_1 B_n(t) \bar X_1^T ,
\end{equation}
where, using the Woodbury identity,
\begin{align*} \label{Bn_def}
B_n(t) &= I_n + n^{-1} \bar X_2^T (t I_p - R_{22})^{-1} \bar X_2 \notag \\
&= t (t I_n - n^{-1} \bar X_2^T \bar X_2 )^{-1}.
\end{align*}
Thus, our key objects are random quadratic forms involving normalized data matrices $\bar X_1$ and $\bar X_2$. The asymptotic properties of these forms are foundational to our results and are presented next.





\section{Results on quadratic forms with normalized entries}
\label{sec:normalized}

In this section, we establish first-order (deterministic) convergence and a CLT for matrix quadratic forms of the type
$n^{-1} \bar X_1 B_n \bar X_1^T$, where $B_n$ is a matrix with bounded spectral norm. While being essential to our purposes, some of the technical developments may be of independent interest; thus, we first present the results in full generality, and then particularize to the context of Model M.

\subsection{First-order convergence} \label{ext_lemma5}



To establish the first-order convergence, we first require some results on bilinear forms involving correlated random vectors of unit length.  A main technical result (see Section \ref{sec_foundations_bound}) is the following:


{\lemma \label{lemma_BS_bound_bilinear}
	Let $B$ be an $n \times n$ non-random symmetric matrix, $x, y \in {\mathbb R}^n$ random vectors of i.i.d. entries with mean zero, variance one, $\E | x_i |^l, \E | y_i |^l  \leq \nu_l$ and $\E [ x_i y_i ] = \rho$. Let $\bar x = \sqrt{n}   x / \| x \|$ and $\bar y = \sqrt{n}  y / \| y \|$. Then, for any $s \geq 1$,
	\begin{equation*}
	\E \left| n^{-1} \bar x^T B \bar y - \rho n^{-1} \tr B \right|^s \leq  {\cal{C}}_s \left[ n^{-s} \left( \nu_{2s} \tr B^s + \left( \nu_4 \tr B^2 \right)^{s/2}  \right) + \| B \|^s \left( n^{-s/2} \nu_4^{s/2} + n^{-s+1} \nu_{2s}  \right)  \right] ,
	\end{equation*} 
	where $ {\cal{C}}_s$ is a constant depending only on $s$.}


This is a generalization of \citet[Lemma~5]{Gao14}, which established a corresponding bound for normalized quadratic forms.  
Lemma \ref{lemma_BS_bound_bilinear} leads to the following first-order convergence result:

%


{\corollary \label{cor_quadForm_as}
	Let $x, y$ be random vectors of i.i.d. entries with mean zero, variance one, $\E | x_i |^{4+\delta}, \E | y_i |^{4+\delta} < \infty$ for some $\delta > 0$, and $\E \left[ x_i y_i \right] = \rho$. Define $\bar x = \sqrt{n}   x / \| x \|$ and $\bar y = \sqrt{n}   y / \| y \|$, and let $B_n$ be a sequence of $n \times n$ symmetric matrices with $\| B_n \|$ bounded. Then,
	\begin{equation*}
     n^{-1} \bar x^T B_n \bar y - n^{-1} \rho \tr B_n \toas 0 .
	\end{equation*} 
}
{\proof Since the $(4+\delta)$th moment and $\|B_n\|$ are bounded, from Lemma~\ref{lemma_BS_bound_bilinear}, 
\begin{equation*}
\E \left| n^{-1} \bar x^T B_n \bar y - n^{-1} \rho \tr B_n \right|^{2+\delta/2} \leq  O(n^{-(1+\delta/4)}).
\end{equation*}
The convergence then follows from Markov's inequality and the Borel-Cantelli lemma.
\endproof}

We now apply this to our Model {\rm M} with \textit{random} matrices
$B_n(\bar X_2)$, independent of $\bar X_1$: 

{\lemma \label{lemma_as}
Assume Model {\rm M} and suppose that $B_n = B_n(\bar X_2)$ is a sequence of random symmetric matrices for which $\| B_n \|$ is $O_{\rm a.s.}(1)$. Then,
\begin{equation*}
n^{-1} \bar X_1 B_n(\bar X_2) \bar X_1^T - n^{-1} \tr B_n(\bar X_2) \Gamma \toas 0. 
\end{equation*} }
{\proof This follows from Fubini's theorem. Specifically, one may use
  the arguments in the proof of [JY, Lemma 5],
applying  Corollary~\ref{cor_quadForm_as},
and noting that $\bar X_1$ is independent of $B_n (\bar X_2)$. 
\endproof}



\subsection{Central Limit Theorem} \label{ext_proposition6}


To establish our main matrix quadratic form CLT result, we first derive a CLT for scalar bilinear forms involving normalized random vectors.  To this end, we must introduce some further notation. Consider zero-mean random vectors $(x, y) \in \mathbb{R}^M \times \mathbb{R}^M$ with
\begin{equation*}
{\rm Cov} \begin{pmatrix}
x \\ y
\end{pmatrix}
=
C =\begin{pmatrix}
	C^{xx} & C^{xy}  \\
	C^{yx} & C^{yy}
	\end{pmatrix} ,
\end{equation*}
where $C_{l l^\prime}^{xy} = \E \, [ x_l y_{l^\prime}]$.
Assume $C_{l l}^{xx} = C_{l l}^{yy} = 1$, i.e., all components of the
$x$ and $y$ vectors have unit variance, and $\rho_l = C_{l l}^{xy} =
\E \, [x_l y_l]$.
We introduce notation for quadratic functions of $x_l, y_l$.
Let $z,w \in \mathbb{R}^M$ with
\begin{equation*}
  z_l = x_l y_l, \qquad  w_l = \rho_l(x_l^2 + y_l^2)/2, \qquad
  C^{zz} = {\rm Cov}(z), \qquad C^{wz} = {\rm Cov}(z,w), \quad {\rm etc.}
\end{equation*}



Let $X = \left( x_{l i} \right)_{M \times n}$ and $Y = \left( y_{l i} \right)_{M \times n}$ be data matrices based on $n$ i.i.d. observations of $(x, y)$, and define the ``normalized'' data matrices  $\bar X = \hat \Sigma_x^{-1/2} X$ and $\bar Y = \hat \Sigma_y^{-1/2} Y$ where $\hat \Sigma_x = {\rm diag} (\hat \sigma_{x_1}^2,\ldots, \hat \sigma_{x_M}^2)$, $\hat \Sigma_y = {\rm diag} (\hat \sigma_{y_1}^2,\ldots, \hat \sigma_{y_M}^2)$,  and $\hat \sigma_{x_l}^2 = n^{-1} \sum_{i=1}^n x_{l i}^2$, $\hat \sigma_{y_l}^2 = n^{-1} \sum_{i=1}^n  y_{l i}^2$.  Introduce notation for the rows $\bar x_{l \cdot }^T$ and $\bar y_{l \cdot }^T$ of the normalized data matrices
\begin{equation*}
\bar X = \left( \bar x_{l i} \right)_{M \times n} =  \begin{bmatrix}
	\bar x_{1 \cdot}^T \\
	\vdots \\
	\bar x_{M \cdot}^T
	\end{bmatrix} , \qquad
\bar Y = \left( \bar y_{l i} \right)_{M \times n} =  \begin{bmatrix}
	\bar y_{1 \cdot}^T \\
	\vdots \\
	\bar y_{M \cdot}^T
	\end{bmatrix} .
\end{equation*}

With this setup, we have the following result, proved in Section
\ref{sec:CLTProof}: 

{\proposition \label{thm_clt}
Let $B_n = (b_{n,ij})$ be random symmetric $n \times n$ matrices, independent of $X, Y$, such that for some finite $\beta$, $\| B_n \| \leq \beta$ for all $n$, and  
	\begin{equation*}
	n^{-1} \sum_{i=1}^n b_{n, ii}^2 \toip  \omega \, , \qquad  n^{-1}  \tr B_n^2 \toip \theta  \, , \qquad   ( n^{-1}  \tr B_n )^2 \toip \phi \,  ,
	\end{equation*}
	all finite.  
	Also, define $Z_n \in {\mathbb R}^M$ with components
	\begin{equation*}
	Z_{n, l} = n^{-1/2} \left[ \bar x_{l \cdot}^T B_n \bar y_{l \cdot} - \rho_l  \tr B_n   \right].
	\end{equation*}
	Then $Z_n  \xrightarrow{\cal D} N_M(0,D)$ with
        \begin{equation}
          \label{D_thm}
          D = (\theta - \omega)J + \omega K_1 + \phi K_2
            = \theta J + \omega K + \phi K_2,
        \end{equation}
	where $K = K_1-J$ and $J, K_1, K_2$ are matrices defined by 
	\begin{align} \label{Kdef}
          J & = C^{xy} \circ C^{yx} + C^{xx} \circ C^{yy} \notag \\
          K_1 & = C^{zz} \\
          K_2 & 
                = C^{ww} - C^{wz} -C^{zw} . \notag
	\end{align}
}
The entries of $K$ are fourth order cumulants of $x$ and
$y$:
\begin{equation}  \label{K_llp}
  K_{ll'}  = \E(x_ly_lx_{l'}y_{l'})
  -\E(x_{l}y_{l}) \E(x_{l'}y_{l'})
  -\E(x_{l}y_{l'}) \E(y_{l}x_{l'})
  -\E(x_{l}x_{l'}) \E(y_{l}y_{l'}).    
\end{equation}
Hence $K$ vanishes if $x, y$ are Gaussian.

The corresponding result with unnormalized vectors is established in
the supplementary manuscript, [JY Theorem 10].
The terms $\theta J + \omega K$ appear in that case, while the
additional term $\phi K_2$ reflects the normalization in
$\bar x_{l \cdot}$ and $ \bar y_{l \cdot}$.
As in [JY], the proof is based on the martingale CLT, in place of the
moment method used in \cite{bai2008},
which stated a similar result for quadratic forms involving unnormalized random vectors.

While potentially of independent interest, Proposition \ref{thm_clt} is important for our purposes
through its application to Model M. 

{\proposition \label{prop_clt}
	
	Assume Model {\rm M} and consider $B_n$ as in Proposition \ref{thm_clt}.  Then,   
	\begin{equation*}
	W_n = n^{-1/2} \left[ \bar X_1 B_n \bar X_1^T - (\tr B_n) \Gamma  \right] \xrightarrow{\cal D} W,
	\end{equation*}
	a symmetric $m \times m$ Gaussian matrix having entries $W_{ij}$ with mean $0$ and covariances
	\begin{equation} \label{cov_prop}
          {\rm Cov}[W_{ij}, W_{i^\prime j^\prime}]
          = \theta  (\kappa_{i j^\prime} \kappa_{j i^\prime} +
            \kappa_{i i^\prime} \kappa_{j j^\prime})
            + \omega  \kappa_{i j i^\prime j^\prime}
            + \phi \check{\kappa}_{iji'j'}.
	\end{equation}
	for $i \leq j$ and $i^\prime \leq j^\prime$.
}
{\proof
The result follows from Proposition~\ref{thm_clt} by
turning the matrix quadratic form $\bar X_1 B_n \bar X_1^T$ into a
vector of bilinear forms, as done also, for example, in [JY,
Proposition 6] and 
\citet[Proposition 3.1]{bai2008}.
Specifically, use an index $l$ for the $M=m(m+1)/2$  pairs $(i,j)$ with $1\leq i \leq j \leq m$. Build the random vectors $(x,y)$ for Proposition~\ref{thm_clt} as follows: if $l = (i,j)$ then set $x_l = \xi_i / \sigma_i$ and $y_l = \xi_j / \sigma_j$. In the resulting covariance matrix $C$ for $(x,y)$, if also $l^\prime = (i^\prime, j^\prime)$,
	\begin{equation*} \label{cov_thm_for_prop}
	C_{l l^\prime}^{xy} = \E [\xi_i \xi_{j^\prime}] / (\sigma_i \sigma_{j^\prime}) = \kappa_{i j^\prime}, \qquad  C_{l l^\prime}^{yx} = \kappa_{j i^\prime}, \qquad C_{l l^\prime}^{xx} = \kappa_{i i^\prime}, \qquad C_{l l^\prime}^{yy} = \kappa_{j j^\prime}
	\end{equation*}
	and, in particular, $\rho_{l} = C_{l l}^{xy} = \kappa_{i j}$ and $\rho_{l^\prime} = \kappa_{i^\prime j^\prime}$, whereas $C_{l l}^{xx} = C_{l l}^{yy} = 1$.
	Component $W_{n,ij}$ corresponds to component $Z_l$ in
        Proposition~\ref{thm_clt}, and we conclude that $W_n
        \xrightarrow{\cal D} W$, where $W$ is a Gaussian matrix with
        zero mean and ${\rm Cov}(W_{ij}, W_{i^\prime, j^\prime}) =
        D_{l l^\prime}$ given by Proposition~\ref{thm_clt}.
It remains to interpret the quantities in \eqref{D_thm} in terms of
Model M.
Substituting $x_l = \bar{\xi}_i$ and $y_l = \bar{\xi}_j$ into
    \eqref{K_llp} and chasing definitions, we get
    $ J_{ll'} = \kappa_{i j'} \kappa_{j i'} + \kappa_{i i'} \kappa_{j j'}$
    and $ K_{ll'} = \kappa_{i j i^\prime j^\prime}$. 
    Observing that $z_l = x_ly_l = \chi_{ij}$ and $w_l =
    \rho_l(x_l^2+y_l^2)/2 = \psi_{ij}$, we similarly find that
    $K_{2,ll'} = \check{\kappa}_{iji'j'}$.
	\endproof}

\section{Proofs of the eigenvalue results} \label{arguments}

In this section we derive the main eigenvalue results, presented in Theorem \ref{thm_main} and Theorem~\ref{thm_main_eigenvec_subcritical}-$(i)$.  


\subsection{Preliminaries}
\label{sec:Prelim}
\textbf{Convergence properties of the eigenvalues of $R_{22}$.} \ 
It is well-known that the empirical spectral density (ESD) of $S_{22}$
converges weakly a.s.\ to the Marchenko-Pastur (MP) law $F_\gamma$,
and that the extreme non-trivial eigenvalues converge to the edges of
the support of $F_\gamma$. For the sample correlation case, 
\citet{Jiang04} shows that the same is true for $R_{22}$. That is, the
empirical distribution of the eigenvalues $\mu_1 \geq \ldots \geq
\mu_p$ of the ``noise'' correlation matrix $R_{22} = n^{-1} \bar X_2
\bar X_2^T$ converges weakly a.s. to the MP law $F_\gamma$, supported
on $[a_\gamma, b_\gamma] = [(1-\sqrt{\gamma})^2, (1+\sqrt{\gamma})^2]$
if $\gamma \leq 1$, and on $\{0\} \cup [a_\gamma, b_\gamma]$
otherwise. Also, the ESD of the $n \times n$ companion matrix $C_n =
n^{-1} \bar X_2^T \bar X_2$, denoted by ${\sf F}_n$, converges weakly
a.s. to the ``companion MP law'' ${\sf F}_\gamma = (1-\gamma) {\bb
  1}_{[0,\infty)} + \gamma F_\gamma$, where ${\bb 1}_{\rm A}$ denotes
the indicator function on set ${\rm A}$. 

It was also shown in \cite{Jiang04} that
\begin{equation} 
\label{Cn_extremes_conv}
\mu_1 \toas b_\gamma \qquad {\rm and} \qquad \mu_{p\wedge n} \toas a_\gamma .
\end{equation}
With these results, if $f_n \to f$ uniformly as continuous functions on the closure ${\mathcal I}$ of a bounded neighborhood of the support of ${\sf F}_\gamma$, then:
\begin{equation} \label{int_law_conv}
\int f_n(x) {\sf F}_n(dx)  \toas  \int f(x) {\sf F}_\gamma(dx) .
\end{equation}
If ${\rm supp}({\sf F}_n)$ is not contained in ${\mathcal I}$, then the left side integral may not be defined. However, 
such event occurs for at most finitely many $n$ with probability one. 


\medskip
\textbf{Almost sure limit of $\hat \ell_{\nu}$.} \ 
The statements in Theorem \ref{thm_main}-$(i)$ and Theorem~\ref{thm_main_eigenvec_subcritical}-$(i)$ follow easily from known results.  Specifically, denote the $\nu$th eigenvalue of the sample covariance $S$ by $\hat \lambda_\nu$. The almost sure limits
\begin{equation}  \label{eq:CovLimits}
\hat \lambda_\nu \toas
\left\{
\begin{array}{c l}	
     \rho_{\nu} , & \ell_\nu > 1 + \sqrt{\gamma} \\
     (1 + \sqrt{\gamma})^2 , &  \ell_\nu \leq 1 + \sqrt{\gamma}
\end{array}\right.
\end{equation}
were established in \cite{baik-silver2006}. From the proof of
\citet[Lemma 1]{ElKaroui2009},  
\begin{equation*}
\max_{i=1,\ldots,m} | \hat \lambda_i - \hat \ell_i | \toas 0 .
\end{equation*}
Therefore the same almost sure limits as \eqref{eq:CovLimits} hold for $\hat{\ell}_{\nu}$.



\medskip
\textbf{High probability events $J_{n \epsilon},  J_{n \epsilon 1}$.} \ 
When necessary, we may confine attention to the event $J_{n \epsilon} = \{ \hat \ell_{\nu} > \min(\rho_\nu, \rho_{\nu n}) -\epsilon, \mu_1 \leq b_\gamma+\epsilon  \}$ or $J_{n \epsilon 1} = \{ \mu_1 \leq b_\gamma+\epsilon  \}$, with $\epsilon>0$ chosen such that $\rho_\nu - b_\gamma \geq 3\epsilon$,  
since from \eqref{eq:SuperSpkEVal} (proven above) and \eqref{Cn_extremes_conv}, these events  occur with probability one for all large $n$.




\medskip
\textbf{Asymptotic expansion of $K(\hat \ell_\nu)$.} \ 
We will establish an asymptotic stochastic expansion for the quadratic form $K(\hat \ell_\nu)$. 
 Specifically, using the decomposition
\begin{equation} \label{eq:K_decomp}
K(\hat \ell_\nu) = K(\rho_{\nu n}) - \left[ K(\hat \ell_\nu) - K(\rho_{\nu n})  \right] \; ,
\end{equation}
we will show that
\begin{equation} \label{Kconv1}
K(\rho_{\nu n})  \toas   - \rho_{\nu} \, {\sf m}( \rho_{\nu}; \gamma) \, \Gamma \, = \, ({\rho_\nu}/{\ell_\nu}) \Gamma 
\end{equation}
and
\begin{equation} \label{eq:delK-firstorder}
K(\hat \ell_\nu) - K(\rho_{\nu n}) = -(\hat \ell_\nu - \rho_{\nu n}) \, [ c(\rho_\nu) \Gamma + o_{\rm a.s.}(1) ] 
\end{equation}
where, for $t \notin {\rm supp}( {\sf F}_\gamma )$,
\begin{equation*} 
{\sf m}(t; \gamma) = \int (x- t)^{-1} {\sf F}_\gamma(dx) , \quad \quad  c( t ) = \int x ( t -x)^{-2} {\sf F}_\gamma (dx).
\end{equation*}
Here ${\sf m}$  is the Stieltjes transform of the companion distribution ${\sf F}_\gamma$.

In establishing \eqref{Kconv1}, start by taking $n$ large enough that $| \rho_{\nu n} - \rho_\nu | \leq \epsilon$, with $\epsilon$ defined as above.  For such $n$, on $J_{n \epsilon 1}$,  we have
\begin{equation*}
\|  B_n (\rho_{\nu n} ) \|  \leq \frac{\rho_\nu + \epsilon }{ \epsilon } \; .
\end{equation*}
Since $J_{n \epsilon 1}$ holds  with probability one for all large $n$, $\|  B_n (\rho_{\nu n} ) \| = O_{\rm a.s.}(1)$, and therefore it follows from  Lemma \ref{lemma_as} that 
\begin{equation*}
K(\rho_{\nu n}) - n^{-1}  \tr B_n(\rho_{\nu n}) \, \Gamma  \toas 0 \; .
\end{equation*}
In addition, \eqref{int_law_conv} yields 
\begin{equation*} \label{trB_limit}
n^{-1} \tr B_n(\rho_{\nu n}) = \int \rho_{\nu n} (\rho_{\nu n}-x)^{-1} {\sf F}_n (dx)  \toas \int \rho_{\nu} (\rho_{\nu}-x)^{-1} {\sf F}_\gamma (dx) =  - \rho_{\nu} \, {\sf m}( \rho_{\nu}; \gamma) \,  .
\end{equation*}
Explicit evaluation gives 
${\sf m}(\rho_{\nu};\gamma) =   -1/\ell_\nu$, [JY, Appendix A],
and \eqref{Kconv1} follows.



To establish \eqref{eq:delK-firstorder}, 
start by recalling that $C_n = n^{-1} \bar{X}_2^T \bar{X}_2$,  and introduce the resolvent notation $Z(t) = (tI_n - C_n)^{-1}$, so that $B_n(t) = t Z(t)$ and $K(t) = n^{-1} \bar X_1 tZ(t) \bar X_1^T$. From the resolvent identity, i.e., $A^{-1} - B^{-1} = A^{-1} (B - A) B^{-1}$ for square invertible $A$ and $B$, and noting that $tZ(t) = C_n Z(t) + I$ from the Woodbury identity, we have, for $t_1, t_2 > b_\gamma$,
\begin{equation*} 
t_1 Z(t_1) - t_2 Z(t_2) = -(t_1-t_2) C_n Z(t_1) Z(t_2)
\end{equation*}
so that
\begin{equation*} \label{diffK_lhat_rho}
K(\hat \ell_\nu) - K(\rho_{\nu n}) = -(\hat \ell_\nu - \rho_{\nu n}) n^{-1} \bar X_1 C_n Z(\hat \ell_\nu) Z(\rho_{\nu n}) \bar X_1^T .
\end{equation*}
Moreover, again by the resolvent identity, $Z(\hat \ell_{\nu}) = Z(\rho_{\nu n}) - (\hat \ell_{\nu} - \rho_{\nu n}) Z(\hat \ell_{\nu}) Z(\rho_{\nu n})$ and we arrive at
\begin{equation} \label{eq:Kdiff2}  
K(\hat \ell_\nu) - K(\rho_{\nu n}) = -(\hat \ell_{\nu} - \rho_{\nu n}) n^{-1} \bar X_1 B_{n1}(\rho_{\nu n},\rho_{\nu n}) \bar X_1^T + (\hat \ell_{\nu} - \rho_{\nu n})^2  n^{-1} \bar X_1  B_{n2}(\hat \ell_{\nu} , \rho_{\nu n}) \bar X_1^T,
\end{equation}
with $B_{nr}(t_1,t_2)$ defined as
\begin{equation} \label{Bnr_def}
B_{nr}(t_1, t_2) = C_n Z(t_1) Z^r(t_2) \; .
\end{equation}
We now characterize the first-order behavior of the two matrix quadratic forms in \eqref{eq:Kdiff2}.  For the first one, we simply mirror the arguments of the proof of \eqref{Kconv1} to obtain  
\begin{equation*} \label{eq:B_n1_as}  
n^{-1} \bar X_1 B_{n1}(\rho_{\nu n},\rho_{\nu n}) \bar X_1^T \toas c(\rho_{\nu}) \Gamma  \; .
\end{equation*}
   For the second one, we again apply similar reasoning, operating on the event $J_{n \epsilon}$. Specifically, it is easy to establish that on $J_{n \epsilon}$, and for $n$ sufficiently large that $| \rho_{\nu n} - \rho_\nu | \leq \epsilon$, $\|  B_{n2}(\hat \ell_{\nu} , \rho_{\nu n}) \|$ is bounded.  Hence, $\|   B_{n2}(\hat \ell_{\nu} , \rho_{\nu n}) \| = O_{\rm a.s.}(1)$, and it follows from  Lemma~\ref{lemma_as} along with \eqref{int_law_conv}  that  
\begin{equation*}
n^{-1} \bar X_1  B_{n2}(\hat \ell_{\nu} , \rho_{\nu n}) \bar X_1^T  = O_{\rm a.s.}(1) .
\end{equation*}
The expansion in \eqref{eq:delK-firstorder} is obtained upon combining the last two displays with \eqref{eq:Kdiff2}.




\medskip
\textbf{CLT of $K(\rho_{\nu n})$.}\
We now specialize Proposition~\ref{prop_clt} to the matrix quadratic
form $K(\rho_{\nu n})$. 


{\proposition \label{prop:CLT_W_n}
	Assume Model {\rm M} and define $\rho_{\nu n}$ by \eqref{rho_nu_def} and  $K(\rho_{\nu n})$ by \eqref{Kt_def}. Then,
	\begin{align*} 
	W_n(\rho_{\nu n}) = \sqrt{n} \left[ K(\rho_{\nu n}) - n^{-1} \tr B_n(\rho_{\nu n}) \Gamma  \right]  \xrightarrow{\cal D} W^{\nu} ,
	\end{align*}
	a symmetric Gaussian random matrix having entries $W_{ij}^{\nu}$ with mean zero and covariance given by 
	\begin{align} \label{cov_prop:CLT_W_n}
          {\rm Cov}[W_{ij}^\nu, W_{i^\prime j^\prime}^\nu]
          &= \frac{\rho_{\nu}^2}{\ell_{\nu}^2 \dot \rho_\nu}
            \left( \kappa_{i j'} \kappa_{j i'} + \kappa_{ii'}
            \kappa_{j j'} \right)
           + \frac{\rho_{\nu}^2}{\ell_{\nu}^2} (\kappa_{iji'j'} +
            \check{\kappa}_{iji'j'}),  
	\end{align}
where $\rho_\nu, \dot \rho_\nu$ are defined in \eqref{rho_nu_def}, and the
terms in parentheses through \eqref{kappacum} and \eqref{kcheckdef}.

}
{\proof   
 Recall $J_{n \epsilon 1} = \{ \mu_1 \leq b_\gamma+\epsilon  \}$ and
 consider $n$ large enough that $\rho_{\nu n} > \rho_\nu -
 \epsilon$. Then, we may apply Proposition~\ref{prop_clt} with $B_n =
 B_n (\rho_{\nu n}) {\bb 1}_{J_{n \epsilon 1}}$, which is independent
 of $\bar{X}_1$, and for which $\| B_n \|$ is bounded. Specifically,
 the result follows by applying Proposition~\ref{prop_clt} to $W_n
 (\rho_{\nu n} ) {\bb 1}_{J_{n \epsilon 1}}$, along with the fact that
 ${\bb 1}_{J_{n \epsilon 1}}  \toas 1$, and particularizing $\omega$,
 $\theta$, and $\phi$ in \eqref{cov_prop}. These quantities, denoted
 respectively by $\omega_\nu$, $\theta_\nu$, and $\phi_\nu$ can be
 computed as in [JY, Appendix A], yielding
 
	\begin{equation*} 
	\omega_\nu = \phi_\nu = \frac{(\ell_{\nu}-1+\gamma)^2}{(\ell_{\nu}-1)^2}  = \frac{\rho_{\nu}^2}{\ell_{\nu}^2} \,\, , \qquad
	\theta_\nu =
          \frac{(\ell_{\nu}-1+\gamma)^2}{(\ell_{\nu}-1)^2-\gamma} =
          \frac{\omega_\nu}{\dot \rho_{\nu}}  .  \qquad \qedhere
	\end{equation*}
}


\textbf{Tightness properties.} \ 
Last but not least, we will establish some tightness properties which will be essential in the derivation of our second-order results.

We first establish a refinement of \eqref{Kconv1}.
Define $K_0(\rho; \gamma) := - \rho \ssm(\rho ; \gamma) \Gamma$, so that \eqref{Kconv1} is rewritten as 
$K(\rho_{\nu n}) \toas K_0(\rho_\nu ; \gamma)$.
Set $g_\rho(x) = \rho (\rho - x)^{-1}$ and write
\begin{equation*}
\tr B_n(\rho) = \sum_{i=1}^n \rho (\rho-\mu_i)^{-1}
= \sum_{i=1}^n g_\rho(\mu_i).
\end{equation*}
Also, introducing
\begin{equation*}
G_n(g) := \sum_{i=1}^n g(\mu_i) - n\int g(x) {\sf{F}}_{\gamma_n}(dx) 
\end{equation*}
we may write
\begin{align}
K(\rho)- K_0(\rho;\gamma_n)
& = n^{-1}[K(\rho) - \tr B_n(\rho) \Gamma] + 
\rho n^{-1} \bigg[\sum_{i=1}^n (\rho-\mu_i)^{-1} -n \int (\rho-x)^{-1} \ssf_{\gamma_n}(dx) \bigg]
\Gamma \notag \\ 
& = n^{-1/2} W_n(\rho) +  n^{-1} G_n(g_\rho) \Gamma. \label{eq:Wdecomp}
\end{align}

\begin{lemma}
	\label{lem:tightness}
	Assume that Model M holds and that $\ell_\nu > 1 + \sqrt
	\gamma$ is simple. For some $b > \rho_1$, let $I$ denote the interval
	$[b_\gamma+3 \epsilon, b]$. Then
	\begin{gather}
	\{ G_n(g_\rho), \rho \in I\} \ \text{is uniformly
		tight}, \label{eq:Gtight} \\
	\{ n^{1/2}[K(\rho)-K_0(\rho ;\gamma_n)], \rho \in I\} \ \text{is uniformly
		tight}, \label{eq:Ktight} \\
	\lhat_\nu-\rhonn = O_p(n^{-1/2}),      \label{eq:evaltight} \\
	a_\nu - p_{\nu} = O_p(n^{-1/2}).   \label{eq:evectight} 
	\end{gather}
\end{lemma}

\begin{proof}
The proofs of \eqref{eq:Gtight}--\eqref{eq:evaltight} appear in
Appendix \ref{sec:proof-lemma-ref}. 	
We show \eqref{eq:evectight} using the expansion $a_\nu - p_\nu = - {\cal R}_{\nu n} D_\nu p_\nu + r_\nu$ given in \eqref{eq:anu-decomp}, from where we recall $\| r_\nu \| = O(\| D_\nu \|^2)$ and note that $\| \mathcal{R}_{\nu n} \| \leq C$ and $D_\nu = K(\hat \ell_{\nu}) - K_0(\rhonn;\gamma_n)$. We then have $a_\nu - p_{\nu}  = O_p( \|D_\nu\| + \|D_\nu\|^2)$.
	Also, from
	\begin{equation*}
		\| D_\nu \|
		\leq \| K(\lhat_\nu) - K(\rhonn)\| + \|K(\rhonn)-K_0(\rhonn;\gamma_n)\|,
	\end{equation*}
	the first term is $O_p(n^{-1/2})$ by \eqref{eq:delK-firstorder} and \eqref{eq:evaltight}, 
	and so is the second term by \eqref{eq:Ktight}. Hence,
	\begin{equation} \label{eq:D_nu_bound}
		\| D_\nu \| = O_p(n^{-1/2})
	\end{equation}
and the proof is completed.
\end{proof}

\subsection{Eigenvalue fluctuations (Theorem~\ref{thm_main}-$(ii)$)}

The proof of Theorem~\ref{thm_main}-$(ii)$ relies on the key expansion
\begin{equation} \label{eq:clt-ready1}
\sqrt n (\lhat_\nu-\rhonn)[1+c(\rho_\nu)\ell_\nu + o_p(1)]
= p_{\nu}^T W_n(\rho_{\nu n})p_{\nu} + o_p(1) ,
\end{equation}
which is obtained by combining the vector equations
$K(\lhat_\nu) a_\nu = \lhat_\nu a_\nu$ and
$K_0(\rhonn; \gamma_n) p_\nu = \rhonn p_\nu$ with
expansions \eqref{eq:Kdiff2} for $K(\lhat_\nu) - K(\rhonn)$ 
and \eqref{eq:Wdecomp} for $K(\rhonn) - K_0(\rhonn;\gamma_n)$. Specifically, we first use $[K(\lhat_\nu) - \lhat_\nu I_m] a_\nu=0$ to write
\begin{equation}   \label{eq:eigenvalue_fluc_proof1}
p_{\nu}^T [K(\lhat_\nu) - \lhat_\nu I_m] p_{\nu}
= (a_\nu-p_{\nu})^T [K(\lhat_\nu) - \lhat_\nu I_m] (a_\nu-p_{\nu})
= O_p(n^{-1}),
\end{equation}
since $\| K(\lhat_\nu) - \lhat_\nu I_m \| = O_p(1)$ by \eqref{eq:K_decomp}-\eqref{eq:delK-firstorder} and \eqref{eq:SuperSpkEVal}, and
$a_\nu-p_{\nu} = O_p(n^{-1/2})$ by Lemma \ref{lem:tightness}.
Also, since $[ K_0(\rhonn; \gamma_n) - \rhonn I_m ] p_{\nu} = 0$, it follows that
\begin{align} \label{eq:eigenvalue_fluc_proof2}
p_{\nu}^T  [K(\lhat_\nu) - \lhat_\nu I_m] p_{\nu}
& = p_{\nu}^T [ K(\lhat_\nu) -  K_0(\rhonn ;\gamma_n) - (\lhat_\nu - \rhonn)I_m] p_{\nu} \notag \\
& = p_{\nu}^T [ K(\lhat_\nu) -  K(\rhonn) - (\lhat_\nu - \rhonn)I_m] p_{\nu}
+ p_{\nu}^T [ K(\rhonn) -  K_0(\rhonn;\gamma_n) ] p_{\nu} \notag \\
& = -(\lhat_\nu-\rhonn)[1 + c(\rho_\nu)\ell_\nu + o_p(1)] + n^{-1/2} p_{\nu}^T W_n(\rho_{\nu n}) p_{\nu} + o_p(n^{-1/2}),
\end{align}
where the last equality follows from \eqref{eq:delK-firstorder}, \eqref{eq:Wdecomp} and \eqref{eq:Gtight}.
Combining \eqref{eq:eigenvalue_fluc_proof1} and \eqref{eq:eigenvalue_fluc_proof2}
gives \eqref{eq:clt-ready1}. 

Asymptotic normality of $\sqrt n (\lhat_\nu-\rhonn)$ now follows from
Proposition \ref{prop:CLT_W_n}, with asymptotic variance
\begin{align*}
  \tilde \sigma_\nu^2
  &= [ 1 + c(\rho_\nu) \ell_\nu ]^{-2} \,\,  {\rm
                      Var} \left[ p_\nu^T W^\nu p_\nu \right]
 =   (\dot \rho_\nu \ell_\nu/ \rho_\nu)^2 \,\, \sum_{i, j, i', j'} \mathcal{P}^\nu_{iji'j'} {\rm Cov} [W_{ij}^\nu , W_{i^\prime j^\prime}^\nu],
\end{align*}
where $W^\nu$ is the $m \times m$ symmetric Gaussian random matrix
defined in Proposition~\ref{prop:CLT_W_n}, with
covariance ${\rm Cov} [W_{ij}^\nu , W_{i^\prime j^\prime}^\nu]$ given
by \eqref{cov_prop:CLT_W_n}. Using this in the developed expression for
the variance above leads to
\begin{equation}
  \label{three}
  \tilde \sigma_\nu^2
  = \dot \rho_\nu \sum_{i,j,i',j'} \mathcal{P}^\nu_{iji'j'}
  (\kappa_{ij'}\kappa_{ji'} + \kappa_{ii'} \kappa_{jj'})
  + \dot \rho_\nu^2 [\mathcal{P}^\nu, \kappa + \check{\kappa}].
\end{equation}
By symmetry and the eigen equation $(\Gamma p_\nu)_i = \sum_j \kappa_{ij} p_{\nu,j} = \ell_\nu p_{\nu,i}$, we have
\begin{align*}
  \sum_{i,j,i',j'} \mathcal{P}^\nu_{iji'j'} \kappa_{ii'} \kappa_{jj'}
  = \sum_{i,j,i',j'} \mathcal{P}^\nu_{iji'j'} \kappa_{ij'} \kappa_{ji'}
  = \sum_{i,j} p_{\nu,i} p_{\nu,j} (\Gamma p_\nu)_i (\Gamma p_\nu)_j
  = \ell_\nu^2 \sum_{i,j} (p_{\nu,i} p_{\nu,j})^2
  = \ell_\nu^2.
\end{align*}
Therefore, the first sum in \eqref{three} reduces to $2 \dot \rho_\nu \ell_\nu^2$,
and so we recover formula \eqref{sigma_thm} of Theorem \ref{thm_main}.

\section{Proofs of the eigenvector results} \label{arguments_eigenvec}
\label{sec_eigenvector}

We now derive the main eigenvector results, presented in Theorem \ref{thm_main_eigenvec}, as well as Theorem~\ref{thm_main_eigenvec_subcritical}-$(ii)$.  


\subsection{Eigenvector inconsistency (Theorem~\ref{thm_main_eigenvec}-$(i)$)}   \label{sec:EvecInconsistency}
The convergence result of Theorem~\ref{thm_main_eigenvec}-$(i)$ follows from two facts: $a_\nu \toas p_\nu$ and $Q_\nu \toas c(\rho_{\nu}) \Gamma$, which will be shown below. Once these facts are established, from \eqref{av_quad_form}, 
\begin{equation*} 
\| \hat p_\nu \|^{-2} \toas p_\nu^T (I_m + c(\rho_{\nu}) \Gamma ) p_\nu = 1 + c(\rho_{\nu}) \ell_{\nu} = \frac{\rho_{\nu}}{\ell_{\nu} \dot \rho_\nu} ,
\end{equation*}
which leads to
\begin{equation*} 
{\rm a.s. \, lim \,} \langle  \hat {\mathfrak p}_\nu,  {\mathfrak p}_\nu  \rangle^2 =
{\rm a.s. \, lim \,} \langle  \hat p_\nu,  p_\nu  \rangle^2 =
{\rm a.s. \, lim \,} \| \hat p_\nu \|^{2} = \frac{\ell_{\nu} \dot \rho_\nu}{\rho_{\nu}} .
\end{equation*}


\medskip
\textbf{Proof of $a_\nu \toas p_\nu$.} \ 
This is a direct consequence of \eqref{eq:anu-decomp} and
\begin{equation*} 
D_\nu = K(\rho_{\nu n}) - (\rho_{\nu n}/\ell_{\nu}) \Gamma + K(\hat \ell_{\nu}) - K(\rho_{\nu n}) \toas 0,
\end{equation*}
which follows from \eqref{Kconv1}, \eqref{eq:delK-firstorder}, and the fact that $\hat \ell_{\nu} - \rho_{\nu n} \toas 0$, given in \eqref{eq:SuperSpkEVal}.


\medskip
\textbf{Proof of $Q_\nu \toas c(\rho_{\nu}) \Gamma$.} \
With $\check Z(t) = (t I_p - R_{22})^{-1}$, we have   
\begin{equation*}
Q_\nu = R_{12} \check Z^2(\rho_{\nu}) R_{21} + R_{12} [ \check Z^2(\hat \ell_{\nu}) - \check Z^2(\rho_{\nu}) ] R_{21} \triangleq Q_{\nu 1} + Q_{\nu 2} .
\end{equation*}
Rewrite $Q_{\nu 1} = n^{-1} \bar X_1 \check B_{n1} \bar X_1^T$, with $\check B_{n1} = n^{-1} \bar X_2^T \check Z^2(\rho_{\nu}) \bar X_2$.  On the high probability event $J_{n \epsilon 1} = \{ \mu_1 \leq b_\gamma + \epsilon \}$, with $\epsilon > 0$ such that $\rho_\nu - b_\gamma \geq 2 \epsilon$, it is easily established that $\| \check B_{n1} \|$ is bounded, and consequently that $\| \check B_{n1} \| = O_{\rm a.s.}(1)$. Hence, Lemma~\ref{lemma_as} can be applied to $Q_{\nu 1}$.  Moreover, noting that
\begin{equation*}
n^{-1} \tr \check B_{n1} =   n^{-1} \tr B_{n 1} ( \rho_\nu, \rho_\nu )  \,  
\end{equation*}
with $B_{n 1}$ defined in \eqref{Bnr_def}, from \eqref{int_law_conv},
\begin{equation*}
n^{-1} \tr \check B_{n1}  \toas \int x (\rho_{\nu}-x)^{-2} {\sf F}_\gamma(dx) = c(\rho_{\nu})  \; .
\end{equation*}
This and Lemma~\ref{lemma_as} imply that $Q_{\nu 1} \toas c(\rho_{\nu}) \Gamma$.   

It remains to show $Q_{\nu 2} \toas 0$. Using a variant of the resolvent identity, i.e.,  $A^{-2} - B^{-2} = - A^{-2} (A^2 - B^2) B^{-2}$ for square invertible $A$ and $B$, we rewrite
\begin{equation*}
Q_{\nu 2} = - 2 (\hat \ell_{\nu} - \rho_{\nu}) n^{-1} \bar X_1 \check B_{n2} \bar X_1^T , 
\end{equation*}
with $\check B_{n2} =  n^{-1} \bar X_2^T \check Z^2(\hat \ell_{\nu}) \left[  \frac12 (\hat \ell_{\nu} + \rho_{\nu} ) I - R_{22} \right] \check Z^2(\rho_{\nu})  \bar X_2$. Working on the high probability event $J_{n\epsilon}$, it can be verified that $\| \check B_{n2} \| = O_{\rm a.s.}(1)$. Thus, Lemma \ref{lemma_as}  together with  \eqref{int_law_conv} imply that $n^{-1} \bar X_1 \check B_{n2} \bar X_1^T = O_{\rm a.s.}(1)$. Since $\hat \ell_{\nu} \toas \rho_{\nu}$, we conclude that $Q_{\nu 2} \toas 0$.


\subsection{Eigenvector fluctuations (Theorem~\ref{thm_main_eigenvec}-$(ii)$)}  \label{sec:EvecFluc}


Again, we use the key expansion \eqref{eq:anu-decomp}.
Since $\| r_\nu \| = O(\| D_\nu \|^2) = O_p(n^{-1})$ from \eqref{eq:D_nu_bound}, we have
\begin{equation*}
\sqrt{n} ( a_\nu  - p_\nu) =  - {\cal R}_{\nu n} \sqrt{n} D_\nu p_\nu + o_p(1).
\end{equation*}  
Furthermore, using a similar decomposition as in the derivation of \eqref{eq:eigenvalue_fluc_proof2},
\begin{align*}
	\sqrt{n} D_\nu &= \sqrt{n} \, [ K(\hat \ell_{\nu}) - K(\rho_{\nu n}) ] + \sqrt{n} \, [ K(\rho_{\nu n}) - K_0(\rho_{\nu n},\gamma_n) ] \\
	&= W_n(\rho_{\nu n}) - \sqrt{n} (\hat \ell_{\nu} - \rho_{\nu n}) c(\rho_{\nu}) \Gamma + o_p(1) ,
\end{align*}
where we have used \eqref{eq:delK-firstorder}, \eqref{eq:Wdecomp}, and \eqref{eq:Gtight}, \eqref{eq:evaltight} of Lemma \ref{lem:tightness}.
With this, and noting that ${\cal R}_{\nu n} \Gamma p_\nu = \ell_\nu {\cal R}_{\nu n} p_\nu = 0$ from the definition of ${\cal R}_{\nu n}$ in \eqref{eq:anu-decomp}, we have 
\begin{equation*}
\sqrt{n} ( a_\nu  - p_\nu) =  - {\cal R}_{\nu n}  W_n(\rho_{\nu n}) p_\nu + o_p(1),
\end{equation*}  
or equivalently,
\begin{equation*}
\sqrt{n} ( P^T a_\nu  - e_\nu) =  - \tilde {\cal R}_{\nu n}  \tilde W_n(\rho_{\nu n}) e_\nu + o_p(1),
\end{equation*}
where
\begin{equation*} 
\tilde {\cal R}_{\nu n} = \frac{\ell_{\nu}}{\rho_{\nu n}} \sum_{k \neq \nu}^{m} (\ell_k - \ell_{\nu})^{-1} e_k e_k^T,
\qquad
\tilde W_n(\rho_{\nu n}) = P^T W_n(\rho_{\nu n}) P.
\end{equation*}
The CLT for $P^T a_\nu$ now follows from Proposition~\ref{prop:CLT_W_n}. In particular, 
\begin{equation*}
\sqrt{n} (P^T a_\nu - e_\nu)  \xrightarrow{\cal D}  \tilde {\cal R}_\nu  w_\nu  \sim   N(0,\Sigma_\nu) ,
\end{equation*}
where $\tilde {\cal R}_\nu = (\ell_\nu/\rho_\nu) \mathcal{D}_\nu$,
recall \eqref{Dnudef}, and
$w_\nu = P^T W^{\nu} p_\nu$ with $W^{\nu}$ defined in
Proposition~\ref{prop:CLT_W_n}.
The covariance matrix
$\Sigma_{\nu} = \tilde {\cal R}_\nu \E [w_\nu w_\nu^T]  \tilde {\cal
  R}_\nu = \mathcal{D}_\nu \tilde{\Sigma}_\nu \mathcal{D}_\nu$
with $\tilde \Sigma_{\nu} = (\ell_\nu/\rho_\nu)^2 \E [w_\nu
w_\nu^T]$.
The $k$th component of $w_\nu$ is given by
$w_\nu (k) = p_k^T W^{\nu}
p_\nu = \sum_{i,j} p_{k,i}  W_{i j}^\nu p_{\nu,j}$ 
and, therefore, 
\begin{equation}  \label{cov_tilde}
  \tilde \Sigma_{\nu, k l}
  = \sum_{i,j,i',j'} p_{k,i} p_{\nu,j} p_{l,i'} p_{\nu,j'} \,\, (\ell_\nu/\rho_\nu)^2{\rm Cov}[W_{i j}^\nu, W_{i' j'}^\nu] \,\, .
\end{equation}
Theorem \ref{thm_main_eigenvec}-($ii$) follows after inserting
\eqref{cov_prop:CLT_W_n} in place of ${\rm Cov}[W_{i j}^\nu, W_{i^\prime
  j^\prime}^\nu]$ and noting that, when $k, l \neq \nu$,
\begin{equation*}
\sum_{i,j,i',j'} p_{k,i} p_{\nu,j} p_{l,i'} p_{\nu,j'} 
(\kappa_{i i^\prime} \kappa_{j j^\prime} + \kappa_{ij'} \kappa_{ji'} )
  = p_k^T \Gamma p_l \cdot p_\nu^T \Gamma p_\nu  +
    p_k^T \Gamma p_\nu \cdot p_\nu^T \Gamma p_l
  = \delta_{kl} \ell_k \ell_\nu \,\, .
\end{equation*}

\subsection{Eigenvector inconsistency in the subcritical case (Theorem~\ref{thm_main_eigenvec_subcritical}-$(ii)$)} \label{sec_eigenvector_subcritical}

From \eqref{av_quad_form} and \eqref{eig_proj_factored}, it suffices to show that $a_\nu^T  Q_\nu a_\nu \toas \infty$ for Theorem~\ref{thm_main_eigenvec_subcritical}-$(ii)$ to hold.
We will establish this by showing that $\lambda_{\rm \min}(Q_\nu) \toas \infty$. 
The approach uses a regularized version of $Q_\nu$
\begin{equation*}
Q_{\nu \epsilon}(t) = R_{12}[(t I_p - R_{22})^2 + \epsilon^2 I_p]^{-1} R_{21},
\end{equation*}
for $\epsilon > 0$. 
Observe that $Q_\nu \succ Q_{\nu \epsilon} (\lhat_\nu)$, so that
\begin{equation*}
\liminf \lmin(Q_\nu)
\geq \liminf \lmin(Q_{\nu \epsilon}(\lhat_\nu))
= \liminf \lmin(Q_{\nu \epsilon}(b_\gamma) + \Delta_{\nu \epsilon}),
\end{equation*}
where $\Delta_{\nu \epsilon} := Q_{\nu \epsilon}(\lhat_\nu)- Q_{\nu \epsilon}(b_\gamma) $ 
(Recall that $\hat{\ell}_\nu \toas b_\gamma$).
We will show that
$\Delta_{\nu \epsilon} \toas 0$,
and
\begin{equation}
\label{eq:qnuep}
Q_{\nu \epsilon}(b_\gamma) \toas
\int x[(b_\gamma-x)^2 + \epsilon^2]^{-1} \ssf_\gamma(dx) \cdot \Gamma
= c_\gamma(\epsilon) \Gamma,
\end{equation}
say. Since $\lmin(\cdot)$ is a continuous function on $m \times m$
matrices, we conclude that
\begin{equation}
\label{eq:cgamma}
\liminf \lmin(Q_\nu) \geq c_\gamma(\epsilon) \lmin(\Gamma),
\end{equation}
and since 
$c_\gamma(\epsilon) \geq  c(b_{\gamma} + \epsilon)$ and $c(b_{\gamma} + \epsilon) \nearrow \infty$
as $\epsilon \searrow
0$ by [JY, Appendix A],
we obtain $\lambda_{\rm \min}(Q_\nu)
\toas \infty$. 
Let us write $Q_{\nu \epsilon}(t) = n^{-1} \bar X_1 \check B_{n \epsilon}(t) \bar X_1$, with
\begin{align*}
\check B_{n \epsilon} (t) & =
n^{-1} \bar X_2^T[(tI_p-n^{-1}\bar X_2 \bar X_2^T)^2+\epsilon^2 I_p]^{-1} \bar X_2 \\
& = H \diag \{ f_\epsilon(\mu_i,t) \} H^T,
\end{align*}
if we write the singular value decomposition of $n^{-1/2} \bar X_2 = V
\mathcal{M}^{1/2} H^T$, with $\mathcal{M}=\diag (\mu_i)_{i=1}^p$, and
define $f_\epsilon(\mu,t) = \mu[(t-\mu)^2+\epsilon^2]^{-1}$.
Evidently $\| \check B_{n \epsilon}(t) \| \leq \epsilon^{-2}\mu_1$ is bounded almost surely.
Thus Lemma \ref{lemma_as} may be applied to $Q_{\nu
	\epsilon}(b_\gamma)$, and since
\begin{equation*}
n^{-1} \tr \check B_{n \epsilon}(b_\gamma) \toas 
 \int f_\epsilon(x, b_\gamma)
 \ssf_\gamma(dx) = c_\gamma(\epsilon)
\end{equation*}
from \eqref{int_law_conv}, our claim \eqref{eq:qnuep} follows. \newline
Consider now $\Delta_{\nu \epsilon}$. 
Fix $a \in \mathbb{R}^m$ such that $\|a\|_2 = 1$, and set $b = n^{-1/2}H^T \bar X_1^T a$.
We have
\begin{equation*}
a^T \Delta_{\nu \epsilon} a
= \sum_{i=1}^p b_i^2 [f_\epsilon(\mu_i,\lhat_\nu) -
f_\epsilon(\mu_i,b_\gamma)]. 
\end{equation*}
Since $|\partial f_{\epsilon}(\mu, t) / \partial t | = |2 \mu (t-\mu)| / [(t-\mu)^2 + \epsilon^2]^2 \leq \mu / \epsilon^3$ for $\mu, \epsilon > 0$ by the arithmetic-mean geometric-mean inequality, we have
\begin{equation*}
| a^T \Delta_{\nu \epsilon} a | 
\leq \mu_1  \epsilon^{-3} |\hat{\ell}_{\nu} - b_{\gamma}| \cdot \|b\|_2^2
 = \mu_1  \epsilon^{-3} |\hat{\ell}_{\nu} - b_{\gamma}| a^T R_{11} a 
\leq \mu_1  \epsilon^{-3} |\hat{\ell}_{\nu} - b_{\gamma}| \lhat_1
\overset{\rm a.s.}{\rightarrow}  0,
\end{equation*}
from Cauchy's interlacing inequality for eigenvalues of symmetric matrices, Theorem~\ref{thm_main}-$(i)$ and Theorem~\ref{thm_main_eigenvec_subcritical}-$(i)$.
Hence $\Delta_{\nu \epsilon} \toas 0$ and the proof of
\eqref{eq:cgamma} and hence of Theorem~\ref{thm_main_eigenvec_subcritical}-$(ii)$ is complete.

\section{Proofs of asymptotic properties of normalized bilinear forms}
\label{sec_foundations}


\subsection{Proof of Lemma \ref{lemma_BS_bound_bilinear} (Trace Lemma)}
\label{sec_foundations_bound}



Lemma \ref{lemma_BS_bound_bilinear} is established by using truncation
arguments, similar to \citet[Lemma~5]{Gao14},  but adapted to
bilinear forms instead of quadratic forms.  Also, in contrast to
that result, we do not consider data that is centered with
the sample mean. 
			
Let $ {\cal C}_s $ denote a constant depending only on $s$, with different instances not necessarily identical.
Define the events ${\cal E}_n^x \triangleq \left\{ \left| n^{-1} \| x \|^2  - 1  \right| \leq \epsilon \right\}$ and ${\cal E}_n^y \triangleq \left\{ \left| n^{-1} \| y \|^2  - 1  \right| \leq \epsilon \right\}$, for some $\epsilon \in (0, 1/2)$, and use $\bar {\cal E}_n^x$, $\bar {\cal E}_n^y$ to denote their complements.
	Using Markov's inequality and Burkholder inequalities for sums of martingale difference sequences \cite[Lemmas~2.13]{BaiSilverstein}, we have, for any $s \geq 1$,
	\begin{align} \label{eq:Burkholder_ineq}
	\Pr \, [ \bar {\cal E}_n^x ]  
	& \leq \epsilon^{-s} \, \E \left| n^{-1} \|x\|^2 - 1 \right|^s = (n\epsilon)^{-s} \, \E \left| \sum_{i=1}^n (x_i^2 - 1) \right|^s \notag \\
	& \leq {\cal C}_s (n \epsilon)^{-s} \left[ \left(  \sum_{i=1}^n \E | x_i^2-1  |^2 \right)^{s/2} + \sum_{i=1}^n \E | x_i^2-1 |^s   \right] 
	= O \left( n^{-s/2} \nu_4^{s/2} + n^{-s+1} \nu_{2s}  \right),
	\end{align}
	and a bound of the same order for $\Pr \, [ \bar {\cal E}_n^y ]$, for the same reason.
	Now define ${\cal E}_n = {\cal E}_n^x \cap {\cal E}_n^y$ and its complement $\bar {\cal E}_n$.
	Then $\Pr \, [ \bar {\cal E}_n ] \leq  \Pr \, [ \bar {\cal E}_n^x ] + \Pr \, [ \bar {\cal E}_n^y ] = O \left( n^{-s/2} \nu_4^{s/2} + n^{-s+1} \nu_{2s}  \right)$ by \eqref{eq:Burkholder_ineq}.
	Also, since $1 = {\bb 1}_{{\cal E}_n} + {\bb 1}_{\bar {\cal E}_n}$(recall that ${\bb 1}_{\rm A}$ denotes the indicator function on set ${\rm A}$),
	we have
	\begin{equation} 	\label{decomp11}
	 \E \left| n^{-1} \bar x^T B \bar y - \rho n^{-1} \tr B \right|^s 
	 = \E \left| n^{-1} \bar x^T B \bar y - \rho n^{-1} \tr B  \right|^s {\bb 1}_{{\cal E}_n }
	   + \E \left| n^{-1} \bar x^T B \bar y - \rho n^{-1} \tr B  \right|^s {\bb 1}_{\bar {\cal E}_n }. 
	\end{equation}
	We now bound the two terms on the right hand side of \eqref{decomp11}. For the second term, from $|n^{-1} \bar x^T B \bar y - \rho n^{-1} \tr B| \leq 2 \|B\|$ and \eqref{eq:Burkholder_ineq}, we have
	\begin{equation*}
	\E \left| n^{-1} \bar x^T B \bar y - \rho n^{-1} \tr B  \right|^s {\bb 1}_{\bar {\cal E}_n }   \leq 2^s \|B\|^s \Pr \, [ \bar {\cal E}_n ]  = \|B\|^s O \left( n^{-s/2} \nu_4^{s/2} + n^{-s+1} \nu_{2s}  \right).
	\end{equation*}
	For the first term in \eqref{decomp11}, use the decomposition
	\begin{equation*}
	n^{-1} \bar x^T B \bar y - \rho n^{-1} \tr B  = \frac{1}{\| x \| \| y \|} \left[ x^T B y - \rho \tr B \right] + \frac{\rho \tr B}{\| x \| \| y \|} \left[ 1 - n^{-1} \| x \| \| y \|  \right] \triangleq a_1 + a_2,
	\end{equation*}
	and the triangle inequality to write
	\begin{equation*}
	\E \left| n^{-1} \bar x^T B \bar y - \rho n^{-1} \tr B \right|^s {\bb 1}_{{\cal E}_n } 
	 \leq {\cal C}_s \left( \E |a_1|^s {\bb 1}_{ {\cal E}_n}  + \E |a_2|^s {\bb 1}_{{\cal E}_n } \right). 
	\end{equation*}
	Noting that $\epsilon \in (0, 1/2)$, $\| x \|^2 \geq n/2$ and $\| y \|^2 \geq n/2$ on ${\cal E}_n$, so that 
	\begin{equation*}
	 \E |a_1|^s {\bb 1}_{ {\cal E}_n}  
	\leq 2^s n^{-s} \E \left| x^T B y - \rho \tr B \right|^s
	\leq {\cal C}_s n^{-s} \left[ \nu_{2s} \tr B^s + ( \nu_4 \tr B^2)^{s/2}  \right], 
	\end{equation*}
	where the last inequality follows from [JY, Lemma 4].
	For $a_2$, for the same reasons and $| \rho| \leq 1$,
	\begin{equation} \label{a2_bound}
	 \E |a_2|^s {\bb 1}_{ {\cal E}_n} 
	 \leq  2^s \left(n^{-1} \tr B \right)^s \E \left| 1 - n^{-1} \|x \| \|y \| \right|^s  {\bb 1}_{{\cal E}_n} \; 
	 \leq 2^s \| B \|^s \E \left| 1 - n^{-1} \|x \| \|y \| \right|^s {\bb 1}_{{\cal E}_n} .   
	\end{equation}
		We now show that $ \E \left| 1 - n^{-1} \|x \| \|y \| \right|^s {\bb 1}_{{\cal E}_n}  = O ( n^{-s/2} \nu_4^{s/2} + n^{-s+1} \nu_{2s}  )$. 
	Note that
	\begin{equation*}
	\left| 1 - n^{-1} \|x \| \|y \| \right| 
	\leq  n^{-1/2} \|y \| \left| n^{-1/2} \|x \| -  1\right|  + \left| n^{-1/2} \|y \| - 1 \right|,
	\end{equation*}
	and that, on ${\cal E}_n$ and with $\epsilon \in (0, 1/2)$, we have $n^{-1/2} \|y \| \leq \sqrt{3/2}$.
	Therefore,
	\begin{align*}
	\E \left| 1 - n^{-1} \|x \| \|y \|  \right|^s {\bb 1}_{{\cal E}_n} 
	&\leq  {\cal C}_s \left[ 
	\E \left| n^{-1/2} \|x \| -  1\right|^s    + \E \left| n^{-1/2} \|y \| - 1 \right|^s 
	\right] \\
	&\leq   {\cal C}_s \left[
	\E \left| n^{-1} \|x \|^2 -  1\right|^s  + \E \left| n^{-1} \|y \|^2 - 1 \right|^s
	\right] =  O \left( n^{-s/2} \nu_4^{s/2} + n^{-s+1} \nu_{2s}  \right), 
	\end{align*}
	by  the fact that $| a -1| \leq | a^2 -1|$ for $a \geq 0$, and \eqref{eq:Burkholder_ineq}.
	Combining this bound with \eqref{a2_bound}, we obtain
	\begin{equation*}
	\E |a_2|^s {\bb 1}_{{\cal E}_n} 
	\leq {\cal C}_s \|B \|^s \left( n^{-s/2} \nu_4^{s/2} + n^{-s+1} \nu_{2s}  \right) .
	\end{equation*}
	The proof is complete after combining the different bounds and inserting them back into \eqref{decomp11}.

\subsection{Proof of Proposition \ref{thm_clt} (CLT)} \label{sec:CLTProof}

We use the Cramer-Wold device and show for each $c \in {\mathbb R}^M$ that $c^T Z_n  \xrightarrow{\cal D} N_M(0,c^T D c)$. 
The proof follows a martingale CLT approach of Baik and
Silverstein presented in the Appendix of \cite{capitaine2009}.   While here
normalized data vectors are considered, a parallel treatment for
bilinear forms with un-normalized data is presented in
the supplementary manuscript [JY, Theorem 10].
	
	
	Start with a single bilinear form $\bar x^T B \bar y = \sum_{i,j} \bar x_i b_{ij} \bar y_j$ built from $n$ vectors ($i=1,\ldots,n$) 
	\[ (\bar x_i, \bar y_i) = \left( \hat \sigma_x^{-1}  x_i, \hat \sigma_y^{-1}  y_i  \right)  \in {\mathbb R}^2 , \]
	where the zero mean i.i.d. vectors $(x_i, y_i)$ have covariance
	\begin{equation*}
	{\rm Cov} \begin{pmatrix}
		x_1\\
		y_1
		\end{pmatrix} =
	\begin{pmatrix}
		1 & \rho  \\
		\rho  & 1
		\end{pmatrix}  ,
	\end{equation*}
	and $\hat \sigma_x^2 =  n^{-1} \, \sum_{i=1}^n x_i^2$ and $\hat \sigma_y^2 = n^{-1} \, \sum_{i=1}^n y_i^2$ are the sample variances. Rewrite $\hat \sigma_x^2 = 1 + v_x$ with $v_x = n^{-1}  \, \sum_{i=1}^n x_i^2 - 1 = O_p(n^{-1/2})$, and use the Taylor expansion of $f(a)=1/\sqrt{1+a}$ around $a=0$ to obtain
	\begin{equation} \label{sigmaExp_x}
	{\hat \sigma_x^{-1}} = 1 - \frac12 v_x + o_p(n^{-1/2}) \, , \qquad  	{\hat \sigma_y^{-1}} =  1 - \frac12 v_y + o_p(n^{-1/2})  .
	\end{equation}
	
	The symmetry of B allows the decomposition
	\begin{align} \label{decomposition}
	n^{-1}  \left( \bar x^T B \bar y - \rho  \tr B \right) =
	n^{-1}  \sum_i ( \bar x_i \bar y_i -\rho ) b_{ii} + \bar x_i S_i(\bar y) + \bar y_i S_i(\bar x) ,
	\end{align}
	where $S_i(\bar y) = \sum_{j=1}^{i-1} b_{ij} \bar y_j$.
	The terms in the sum above are not martingale differences, since the data vectors $\bar x$, $\bar y$ are normalized to unit length. In order to apply the Baik-Silverstein argument, we aim at finding an alternative decomposition in terms of the unnormalized data vectors $x, y$; let us see this, term by term. For the first term, using \eqref{sigmaExp_x},
	\begin{align} \label{term1_1}
	n^{-1} \sum_i ( \bar x_i \bar y_i -\rho ) b_{ii}
	&= ({\hat \sigma_x \hat \sigma_y} \, n)^{-1}  \sum_i x_i y_i b_{ii} - n^{-1} \sum_i \rho \, b_{ii}  \notag \\
	&= \left[ 1 - \frac12 (v_x + v_y) + o_p(n^{-1/2})  \right] n^{-1} \sum_i x_i y_i b_{ii} - n^{-1} \sum_i \rho \, b_{ii} \notag \\
	&= n^{-1} \sum_i (x_i y_i - \rho) b_{ii} - \frac12 (v_x + v_y) \, n^{-1} \sum_i x_i y_i b_{ii}  + o_p(n^{-1/2}) . \quad
	\end{align}
	Note that
	\begin{equation*}
	n^{-1} \sum_i x_i y_i b_{ii} =  	n^{-1} \sum_i \rho \, b_{ii}  +  	n^{-1} \sum_i ( x_i  y_i -\rho ) b_{ii} =  	n^{-1} \sum_i \rho \, b_{ii} + O_p(n^{-1/2}) 
	\end{equation*}
	and recall that $v_x$ and $v_y$ are $O_p(n^{-1/2})$ so that, from \eqref{term1_1},
	\begin{align*} 
	n^{-1} \sum_i ( \bar x_i \bar y_i -\rho ) b_{ii}
	&= 	n^{-1} \sum_i (  x_i  y_i -\rho ) b_{ii} - \frac12 \rho \, (n^{-1}  \tr B) \, (x_i^2 + y_i^2 - 2)  + o_p(n^{-1/2}) .
	\end{align*}
	
	For the second term in \eqref{decomposition},
	\begin{equation*}
	n^{-1} \sum_i \bar x_i S_i(\bar y)  = ({\hat \sigma_x \hat \sigma_y})^{-1} \, n^{-1} \sum_i x_i S_i( y) 
	\end{equation*}
	where, from the independence of $x_i y_j$ and $b_{ij}$ and the spectral norm bound of $B$,
	\begin{equation*}
	n^{-1} \sum_i  x_i S_i( y)  = n^{-1} \sum_i  \sum_{j=1}^{i-1}  x_i y_j b_{ij} = O_p(n^{-1/2}) .
	\end{equation*}
	This, along with the fact that $v_x$ and $v_y$ are $O_p(n^{-1/2})$, yield
	\begin{align*}
	n^{-1}  \sum_i \bar x_i S_i(\bar y)  &= \left[ 1 - \frac12 (v_x + v_y) + o_p(n^{-1/2})  \right] \, n^{-1}  \sum_i  x_i S_i( y) \\
	&= n^{-1}  \sum_i  x_i S_i( y) + o_p(n^{-1/2}) .
	\end{align*}
	The third term in \eqref{decomposition}, $n^{-1}  \sum_i  \bar y_i S_i(\bar x)$, is handled similarly.
	Altogether, we have the decomposition
	\begin{equation*}
	n^{-1}  \left( \bar x^T B \bar y - \rho  \tr B \right) =
	n^{-1}  \sum_i (  x_i  y_i -\rho ) b_{ii} - \frac12 \rho \, (n^{-1}  \tr B) \, (x_i^2 + y_i^2 - 2)
	 +  x_i S_i( y) +  y_i S_i( x) + o_p(n^{-1/2}),
	\end{equation*} 
	where we can now apply the Baik-Silverstein argument.
	Specifically, in the setting of the theorem,
	\begin{equation*} 
	c^T Z_n = n^{-1/2} \sum_{l} c_l \left( \bar x_{l \cdot}^T B_n \bar y_{l \cdot} - \rho_l  \tr B_n \right) 
	= \sum_{i=1}^n Z_{di} + Z_{yi} + Z_{xi} + o_p(1)  = \sum_{i=1}^n Z_{ni} + o_p(1) ,
	\end{equation*} 
	where 
	\begin{align} 
	\sqrt{n} \, Z_{di} &= \sum_{l} c_l \left[   (  x_{l i}  y_{l
                             i} -\rho_l ) b_{ii} - \tfrac12 \rho_l \,
                             (n^{-1}  \tr B_n) \, (x_{l i}^2 + y_{l
                             i}^2 - 2)  \right]
                             = \sum_l c_l \, \mathsf{b}_i^T \mathsf{z}_{li} 
                             \notag \\
	\sqrt{n} \, Z_{yi} &=  \sum_{l} c_l    x_{l i} S_i( y_{l \cdot})  \notag \\
	\sqrt{n} \, Z_{xi} &=  \sum_{l} c_l    y_{l i} S_i( x_{l \cdot}) \notag
	\end{align} 
	are martingale differences w.r.t. ${\mathcal F}_{n,i}$, the
        $\sigma$-field generated by $B_n$ and $\{ (x_{lj}, y_{lj}), 1
        \leq l \leq M ,  1 \leq j \leq i \}$.
        In the case of $Z_{d,i}$ we have introduced notation
        \begin{equation*}
          \bar{b} = n^{-1} \tr B_n, \qquad
          \mathsf{b}_i =  \begin{pmatrix}
            b_{ii} \\ - \bar{b}
          \end{pmatrix}, \qquad
          \mathsf{z}_{li} =
          \begin{pmatrix}
            z_{li} - \rho_l \\ w_{li} - \rho_l
          \end{pmatrix},
        \end{equation*}
        recalling that $z_{li} = x_{li} y_{li}$ and $w_{li} = \rho_l (x_{li}^2+y_{li}^2)/2$.

        Let $\E_{i-1}$ denote conditional expectation w.r.t. ${\mathcal F}_{n,i-1}$ and apply the martingale CLT. The limiting variance is found from $v^2 = {\rm plim}\, V_n^2$ with
	\begin{align} \label{Vn} 
	V_n^2 = \sum_{i=1}^n  \E_{i-1} [Z_{ni}^2] = V_{n,dd} + 2 (V_{n,dy} + V_{n,dx}) + V_{n,yy} + V_{n,xx} + 2 V_{n,xy} ,
	\end{align}
	where $V_{n,ab} = \sum_{i=1}^n \E_{i-1} [Z_{ai} Z_{bi}]$ for
        indices $a,b \in \{d,y,x\}$. The terms $Z_{yi}$ and $Z_{xi}$
        are exactly as in [JY] and, therefore
	\begin{align*}
	V_{n,yy} = V_{n,xx} & \xrightarrow{p} \frac12 (\theta - \omega) c^T C^{xx} \circ C^{yy} c \notag \\
	V_{n,xy} & \xrightarrow{p} \frac12 (\theta - \omega) c^T C^{xy} \circ C^{yx} c . 
	\end{align*}
	We only need to compute $V_{n,dd}$, $V_{n,dx}$ and
        $V_{n,dy}$. Start with $V_{n,dd} = \sum_{i=1}^n  \E_{i-1} [Z_{di}^2]$,
	where
        \begin{equation*}
          n \E_{i-1} Z_{di}^2 = \sum_{l , l^\prime} c_l c_{l^\prime}
          \mathsf{b}_i^T \E_{i-1} (\mathsf{z}_{li} \mathsf{z}_{li}^T)
          \mathsf{b}_i, \qquad \text{and} \qquad
            \E_{i-1} (\mathsf{z}_{li} \mathsf{z}_{li}^T)
     =     \begin{pmatrix}
       C^{zz}_{ll'} & C^{zw}_{ll'} \\
       C^{wz}_{ll'} & C^{ww}_{ll'}
     \end{pmatrix}
        \end{equation*}
does not depend on $i$. Consequently
\begin{equation*}
  V_{n,dd} = \sum_{l , l^\prime} c_l c_{l^\prime} \Big[ (n^{-1} \sum_i
  b_{ii}^2) C^{zz}_{ll'} + \bar{b}^2 (C^{ww}_{ll'} - C^{wz}_{ll'} -
  C^{zw}_{ll'}) \Big]
\end{equation*}
and ${\rm plim}\, V_{n,dd} = c^T \left( \omega K_1 + \phi K_2  \right)
c$, with $K_1, K_2$ given by \eqref{Kdef}.

	Turn now to
	\begin{align} \label{Vdy_1}
	V_{n,dy} &= \sum_{i=1}^n  \E_{i-1} [Z_{di} Z_{yi}] 
	\notag \\ &
	= \sum_{l , l^\prime} c_l c_{l^\prime} M_{l, l^\prime}^{(1)} \, \left[ n^{-1} \sum_i b_{ii} S_i(y_{l \cdot}) \right] -  \sum_{l , l^\prime} c_l c_{l^\prime} M_{l, l^\prime}^{(2)} \, \left[  n^{-1} \sum_i S_i(y_{l \cdot}) \right] ,
	\end{align}
	where
	\begin{equation*} 
	M_{l, l^\prime}^{(1)} = \E [(x_l y_l - \rho_{l}) x_{l^\prime}], \quad
	M_{l, l^\prime}^{(2)} = \frac12 \rho_{l} (n^{-1}  \tr B_n)  \E [(x_l^2 + y_l^2 - 2) x_{l^\prime}] .
	\end{equation*}
By [JY, Lemma 12], the two quantities between brackets in
\eqref{Vdy_1} converge to zero 
in probability and, therefore, $V_{n,dy} \xrightarrow{p}
0$. Similarly, $V_{n,dx} \xrightarrow{p} 0$. Combining terms according
to \eqref{Vn} and the previous limits, we finally get $v^2 = c^T D c$,
with $D$ as in the theorem. 

        Finally, we verify the Lindeberg condition.
        An important closure property, shown in
        \citet[Appendix]{capitaine2009}, called [A] below,
        states that, for random variables $X_1, X_2$ and positive $\epsilon$,
	\begin{equation*} 
    \E [ |X_1+X_2|^2 {\bf 1}_{|X_1+X_2|\geq \epsilon}] \leq 4 \left(   \E [ |X_1|^2 {\bf 1}_{|X_1|\geq \epsilon/2}] +  \E [ |X_2|^2 {\bf 1}_{|X_2|\geq \epsilon/2}]   \right) .
	\end{equation*}
	From this, it suffices to establish the Lindeberg condition for the martingale difference sequences
	\begin{align*}
	Z_{l i}^{(1)} = \frac{1}{\sqrt{n}} (  x_{l i}  y_{l i} -\rho_l ) b_{ii} , \quad
	Z_{l i}^{(2)} &=   \frac{\rho_l}{2\sqrt{n}}   (n^{-1}  \tr B_n) \, (x_{l i}^2  - 1) , \quad
	Z_{l i}^{(3)} =  \frac{\rho_l}{2\sqrt{n}}  (n^{-1}  \tr B_n) \, (y_{l i}^2 - 1) , \\
   & Z_{l i}^{(4)} =	 \frac{1}{\sqrt{n}} x_{l i} S_i( y_{l \cdot}) , \quad
   Z_{l i}^{(5)} =	 \frac{1}{\sqrt{n}} y_{l i} S_i( x_{l \cdot}) . \notag
	\end{align*}
	This follows just as in [A];
        recalling that $\|B\| \leq \beta$ we have, for
        $\epsilon > 0$,
        
	\begin{equation*}
    \sum_{i=1}^n \E  [ |Z_{l i}^{(1)}|^2 {\bf 1}_{|Z_{l i}^{(1)}|\geq \epsilon}] \leq \beta^2 \E  [ (  x_{l i}  y_{l i} -\rho_l )^2 {\bf 1}_{|x_{l i}  y_{l i} -\rho_l |\geq \sqrt{n} \epsilon /\beta}]  \to 0
    \end{equation*}	
	as $n \to \infty$, by the dominated convergence theorem. The
        sequences $Z_{l i}^{(2)}$ and $Z_{l i}^{(3)}$ are handled
        analogously, with $(x_{l i}^2  - 1)$ and $(y_{l i}^2  - 1)$ in
        place of $(  x_{l i}  y_{l i} -\rho_l )$. For $Z_{l i}^{(4)}$,
        it can be easily shown, as in [A],
        that $\E [ | S_i( y_l) |^4] = O(1)$, so that $\E [ | Z_{l i}^{(4)} |^4] = O(n^{-2})$ and
    \begin{equation*}
	\sum_{i=1}^n \E  [ |Z_{l i}^{(4)}|^2 {\bf 1}_{|Z_{l i}^{(4)}|\geq \epsilon}] \leq 
	(1/\epsilon^2) \sum_{i=1}^n \E  |Z_{l i}^{(4)}|^4 \to 0
	\end{equation*}	
	as $n \to \infty$. The same reasoning applies to the last sequence $Z_{l i}^{(5)}$, and the verification is complete.

\section*{Acknowledgements}
This work was supported in part by
NIH R01 EB001988 (IMJ, JY), by the Hong Kong RGC General Research Fund 16202918 (MRM, DMJ), and by a Samsung Scholarship (JY).

\newpage
{\appendices

\section{Gaussian particularizations}
 
\subsection{Proof of Corollary~\ref{cor_Gaussian}}
\label{cor_Gaussian_proof}

Refer to definition \eqref{kcheckdef} of $\check{\kappa}$, and
introduce also
\begin{equation*}
  \check{\kappa}_{1,iji'j'} = {\rm Cov}(\psi_{ij},\psi_{i'j'}), \qquad
  \check{\kappa}_{2,iji'j'} = {\rm Cov}(\psi_{ij},\chi_{i'j'}).
\end{equation*}
Note that the centering terms in covariances $\check{\kappa}_{1}$ and
$\check{\kappa}_{2}$ are respectively
$\E \psi_{ij} \E \psi_{i'j'}$ and $\E \psi_{ij} \E \chi_{i'j'}$, and that
both equal $\kappa_{ij} \kappa_{i'j'}$.
Observe also that the sum
\begin{equation*}
  [\mathcal{P}^\nu, \check{\kappa}_{2}]
  = \sum_{i,j,i',j'}  {\cal P}_{i j i' j'}^\nu {\rm Cov}(\psi_{ij},\chi_{i'j'})
\end{equation*}
is unchanged by swapping the indices $(ij)$ with $(i' j')$, so
\begin{equation}   \label{Pkappa}
  [\mathcal{P}^\nu, \check{\kappa}] = [\mathcal{P}^\nu,
  \check{\kappa}_{1}] - 2 [\mathcal{P}^\nu, \check{\kappa}_{2}] \; .
\end{equation}
Using definition \eqref{chipsidef} of $\psi_{ij}$ and setting $\mu_{iji'j'} = \E [ \bar{\xi}_i \bar{\xi}_j \bar{\xi}_{i'}
\bar{\xi}_{j'} ]$, 
\begin{equation*}
  \sum_{i,j,i',j'}  {\cal P}_{i j i' j'}^\nu \E(\psi_{ij} \psi_{i'j'})
  = \frac{1}{4} \sum_{i,j,i',j'}  {\cal P}_{i j i' j'}^\nu
  \kappa_{ij} \kappa_{i'j'} \E(\bar{\xi}_{i}^2 +
  \bar{\xi}_{j}^2)(\bar{\xi}_{i'}^2 + \bar{\xi}_{j'}^2) 
  = \sum_{i,j,i',j'}  {\cal P}_{i j i' j'}^\nu \kappa_{ij}
  \kappa_{i'j'} \mu_{iii'i'}  ,
\end{equation*}
since the latter sum is unaffected by replacing $i$ with $j$, and $i'$
with $j'$. 

Now we can insert the centering in $\check{\kappa}_{1}$, and argue in
a parallel manner for $\check{\kappa}_{2}$, to obtain
\begin{align*}
  [\mathcal{P}^\nu,  \check{\kappa}_{1}]
  & = \sum_{i,j,i',j'}  {\cal P}_{i j i' j'}^\nu \kappa_{ij}
    \kappa_{i'j'} (\mu_{iii'i'} - 1)
    = \ell_\nu^2 \sum_{i, i'} (p_{\nu,i})^2 (p_{\nu,{i'}})^2 (\mu_{iii'i'} - 1),
  \\
  [\mathcal{P}^\nu,  \check{\kappa}_{2}]
  & = \sum_{i,j,i',j'}  {\cal P}_{i j i' j'}^\nu \kappa_{ij}
     (\mu_{iii'j'} - \kappa_{i'j'})
    = \ell_\nu \sum_{i, i', j'} (p_{\nu,i})^2 p_{\nu,i'} p_{\nu,{j'}}  (\mu_{iii'j'} - \kappa_{i'j'}),
\end{align*}
where in each final equality we used the summation device for indices
occurring exactly twice; for example, if $j$ appears twice, $\sum_j \kappa_{ij} p_{\nu,j} = \ell_\nu p_{\nu,i}$.

To this point, no Gaussian assumption was used. 
If the data is Gaussian, $\mu_{i j i^\prime  j^\prime } = \kappa_{i
  j} \kappa_{i^\prime j^\prime} + \kappa_{i i^\prime} \kappa_{j
  j^\prime} + \kappa_{i j^\prime} \kappa_{j i^\prime}$, and in
particular $\mu_{iii'i'}-1 = 2 \kappa_{ii'}^2$, and
$\mu_{iii'j'}-\kappa_{i'j'} = 2 \kappa_{ii'} \kappa_{ij'}$.
We then obtain
\begin{align*}
  [\mathcal{P}^\nu,  \check{\kappa}_{1}]
  & = 2 \ell_\nu^2 \sum_{i, i'} (p_{\nu,i} \kappa_{ii'} p_{\nu,{i'}})^2
    = 2 \ell_\nu^2 \tr (P_{D,\nu} \Gamma P_{D,\nu})^2, \\
    [\mathcal{P}^\nu,  \check{\kappa}_{2}]
  & = 2 \ell_\nu^3 \sum_{i} (p_{\nu,i})^4
    = 2 \ell_\nu^3 \tr (P_{D,\nu}^4),
\end{align*}
and inserting these into \eqref{Pkappa} and then \eqref{sigma_thm}, we
complete the proof.

\subsection{Proof of Corollary~\ref{cor_eigenvec_Gaussian}}
\label{cor_eigenvec_Gaussian_proof}

If the data is Gaussian, $\kappa_{i j i' j'} = 0$ and $\E [ \bar{\xi}_i \bar{\xi}_j \bar{\xi}_{i^\prime}
\bar{\xi}_{j^\prime} ] = \kappa_{i j} \kappa_{i^\prime j^\prime} + \kappa_{i i^\prime} \kappa_{j j^\prime} + \kappa_{i j^\prime} \kappa_{j i^\prime}$. Hence,
\begin{align}
	(\tilde \Sigma_\nu)_{k l }
	&=   \dot \rho_{\nu}^{-1}  \ell_k \ell_{\nu} \,
	\delta_{k,l} +  [\mathcal{P}^{k \nu l \nu},
	\check{\kappa}], \label{Sigma-til_gaussian}
\end{align}
and, from the definition of $\check \kappa$ in \eqref{chipsidef}-\eqref{kcheckdef}, we obtain
\begin{align}
\check{\kappa}_{iji'j'} =  \frac12  \kappa_{i j} \kappa_{i^\prime j^\prime} ( \kappa_{i i^\prime}^2 + \kappa_{j j^\prime}^2 + \kappa_{i j^\prime}^2 + \kappa_{i^\prime j}^2  ) 
	-  \kappa_{i^\prime j^\prime}  ( \kappa_{i i^\prime} \kappa_{j i^\prime} + \kappa_{i j^\prime} \kappa_{j j^\prime} ) - \kappa_{i j}  ( \kappa_{i i^\prime} \kappa_{i j^\prime} + \kappa_{i^\prime j} \kappa_{j j^\prime} ) \,  .
	\label{cov_gaussian_xx}
\end{align}
With this, we evaluate $ [\mathcal{P}^{k \nu l \nu}, \check{\kappa}]$, where we need to deal with terms of two types: involving products such as $\kappa_{i j} \kappa_{i^\prime j^\prime} \kappa_{i i^\prime}^2$ and such as $ \kappa_{i^\prime j^\prime}  \kappa_{i i^\prime} \kappa_{j i^\prime}$. For the first type, we use the same summation device as in Appendix \ref{cor_Gaussian_proof}, i.e., if for example index $j$ occurs exactly twice, $\sum_j \kappa_{ij} p_{\nu,j} = \ell_\nu p_{\nu,i}$, so that
\begin{align*}
 \sum_{i,j,i^\prime,j^\prime} {\cal P}_{i j i^\prime j^\prime}^{k \nu l \nu} \, \kappa_{i j} \kappa_{i^\prime j^\prime} \kappa_{i i^\prime}^2   =
  \ell_{\nu}^2  \sum_{i,i^\prime} p_{k,i} (    p_{\nu,i}  \kappa_{i i^\prime}^2  p_{\nu,i'} ) p_{l,i'}  =   \ell_{\nu}^2 \, p_k^T P_{D,\nu} (\Gamma \circ \Gamma) P_{D,\nu} p_l =
  \ell_{\nu}^2 {\cal Z}_{k l} ,
\end{align*}
where, recall $P_{D,\nu} = {\rm diag}(p_{\nu,1},\ldots,p_{\nu,m})$. For the second type of terms, by the same device,
\begin{align*}
\sum_{i,j,i^\prime,j^\prime} {\cal P}_{i j i^\prime j^\prime}^{k \nu l \nu} \,  \kappa_{i^\prime j^\prime}  \kappa_{i i^\prime} \kappa_{j i^\prime}  =
\ell_{k} \ell_{\nu}^2  \sum_{i'} p_{k,i'}  p_{\nu,i'}^2 p_{l,i'}  =   \ell_k \ell_{\nu}^2 \, p_k^T P_{D,\nu}^2 p_l =
  \ell_k \ell_{\nu}^2 {\cal Y}_{k l}  .
\end{align*}
The rest of terms are evaluated similarly, yielding
\begin{align*}
&\sum_{i,j,i^\prime,j^\prime} {\cal P}_{i j i^\prime j^\prime}^{k \nu l \nu} \, \kappa_{i j} \kappa_{i^\prime j^\prime} \kappa_{j j^\prime}^2 =  \ell_{k} \ell_l {\cal Z}_{k l},  \qquad 
\sum_{i,j,i^\prime,j^\prime} {\cal P}_{i j i^\prime j^\prime}^{k \nu l \nu} \,  \kappa_{i^\prime j^\prime}  \kappa_{i j^\prime} \kappa_{j j^\prime}   =  \ell_k \ell_l \ell_{\nu} {\cal Y}_{k l},   \\
&\sum_{i,j,i^\prime,j^\prime} {\cal P}_{i j i^\prime j^\prime}^{k \nu l \nu} \, \kappa_{i j} \kappa_{i^\prime j^\prime} \kappa_{i j^\prime}^2  =  \ell_{\nu} \ell_l {\cal Z}_{k l} , \qquad
\sum_{i,j,i^\prime,j^\prime} {\cal P}_{i j i^\prime j^\prime}^{k \nu l \nu} \,  \kappa_{i j} 
\kappa_{i i^\prime} \kappa_{i j^\prime}  =  \ell_l \ell_{\nu}^2 {\cal Y}_{k l} , \\
&\sum_{i,j,i^\prime,j^\prime} {\cal P}_{i j i^\prime j^\prime}^{k \nu l \nu} \, \kappa_{i j} \kappa_{i^\prime j^\prime} \kappa_{i^\prime j}^2 =  \ell_{\nu} \ell_k {\cal Z}_{k l}, \qquad \sum_{i,j,i^\prime,j^\prime} {\cal P}_{i j i^\prime j^\prime}^{k \nu l \nu} \,  \kappa_{i j} 
\kappa_{i^\prime j} \kappa_{j j^\prime}  =  \ell_k \ell_l \ell_{\nu} {\cal Y}_{k l} .
\end{align*}
Combining terms according to \eqref{Sigma-til_gaussian}-\eqref{cov_gaussian_xx} leads to the result of Corollary~\ref{cor_eigenvec_Gaussian}.

\section{Proof of Lemma \ref{lem:tightness},
  \eqref{eq:Gtight}--\eqref{eq:evaltight}}  
\label{sec:proof-lemma-ref}

	Tightness of $G_n(g_\rho)$ in \eqref{eq:Gtight}, and \textit{a fortiori} that of
	$n^{-1/2} G_n(g_\rho)$, follows from that of $\widehat{M}_n(z)$ in \citet[Proposition 1]{Gao14}, itself an adaptation of Lemma 1.1 of \citet{basi04} to the sample correlation setting, and the arguments following that Lemma. Note in particular that, with notation $x_r, \mathcal{C}, \bar{\mathcal{C}}$ from \cite{basi04}, a complex contour $\mathcal{C}\cup \bar{\mathcal{C}}$ enclosing the support of ${\sf F}_\gamma$ can be chosen, by taking $b_\gamma < x_r < b_\gamma + 3\epsilon$, such that $| g_\rho(z) |$ is bounded above by a constant for $z \in \mathcal{C}\cup \bar{\mathcal{C}}$. 

	To prove \eqref{eq:Ktight}, from \eqref{eq:Wdecomp} it suffices to show that the matrix valued process $\{ W_n(\rho) \in \mR^{m \times m}, \rho \in I \}$ is uniformly tight.
	Since $m$ stays fixed throughout, we only need to show tightness for each of the scalar processes formed from the matrix entries $e_k^T W_n(\rho) e_l$ on $I$.
	
	Let $\mP_n, \E_n$ denote probability and expectation conditional on the event $E_{n  \epsilon} = \{\mu_1 \leq b_\gamma + \epsilon \}$. 
	We show tightness of $W_n(\rho)$ on $I$ by establishing the moment criterion of \citet[eq. (12.51)]{Bill68}:
	we exhibit $C$ such that for each
	$k,l \leq m$ and
	$\rho, \rho' \in I$,
	\begin{equation*}
	\E_n |e_k^T[W_n(\rho) - W_n(\rho') ]e_l|^2 \leq C(\rho-\rho')^2.
	\end{equation*}
	Write the quadratic form inside the expectation as $\bx^T \check B_n \by - \kappa_{kl} \tr \check B_n$ with
	$\bx = \bar X_1^T e_k$ and $\by = \bar X_1^T e_l$ being the
	$k^{th}$ and $l^{th}$ rows of $\bar X_1$
	and $\check B_n = n^{-1/2}[B_n(\rho)-B_n(\rho')]$.
	Lemma \ref{lemma_BS_bound_bilinear} with $p=2$ yields
	\begin{equation*}
	\E_n |e_k^T[W_n(\rho) - W_n(\rho') ]e_l|^2
	\leq 2 C_2 \nu_4 \E_n [\tr \check B_n^2 + \| n^{1/2} \check B_n \|^2].
	\end{equation*}
	Now $n^{1/2} \check B_n$ has eigenvalues 
	$(\rho' - \rho) \mu_i  (\rho' - \mu_i)^{-1} (\rho - \mu_i)^{-1}$, 
	so that on $E_{n \epsilon}$ we have $\tr \check B_n^2 \leq \| n^{1/2} \check B_n \|^2 \leq C(\rho'-\rho)^2$, which establishes the
	moment condition.
	
	To establish \eqref{eq:evaltight}, we work conditionally on $E_{n \epsilon}$. 
	The tightness just established yields, for given $\epsilon$, a value $M$ for which the event $E_n'$ defined by
	\begin{equation*}
	\sup_{\rho\in I} n^{1/2} \| K(\rho) - K_0(\rho;\gamma_n) \| > \hf M
	\end{equation*}
	has $\mP_n$-probability at most $\epsilon$.
	For all large enough $n$ such that $b_\gamma + 3\epsilon > (1 + \sqrt{\gamma_n})^2$, we combine this with the eigenvalue perturbation bound 
	\begin{equation} \label{eq:evalpbd}
	|\lambda_\nu(\rho) - \lambda_{0 \nu}(\rho)|
	\leq \| K(\rho)-K_0(\rho;\gamma_n) \|
	\end{equation}
	for  $\rho \in I$, 
	where $\lambda_\nu(\rho)$ and $\lambda_{0 \nu}(\rho) = - \rho \ssm(\rho;\gamma_n)\ell_\nu - \rho $ are the
	$\nu^{\text{th}}$ eigenvalues of $K(\rho) - \rho I_m$ and $K_0(\rho;\gamma_n) - \rho I_m$ respectively.
	Observe that $\lambda_{0 \nu}(\rhonn) = 0$ and
	\begin{equation*}
	\partial_\rho \lambda_{0 \nu}(\rho) = - 1 - \ell_{\nu} \int x(\rho-x)^{-2}
	{\sf{F}}_{\gamma}(dx) < -1,
	\end{equation*}
    hence for $\rho_{n\pm} = \rhonn \pm M n^{-1/2}$,
	we have $ \lambda_{0 \nu}(\rho_{n-}) \geq Mn^{-1/2}$ and
	$ \lambda_{0 \nu}(\rho_{n+}) \leq -Mn^{-1/2}$.
	Now \eqref{eq:evalpbd} shows that on event $E_n'^{c}$,
	$\lambda_{\nu}(\rho_{n-}) \geq \hf Mn^{-1/2}$ and 
	$\lambda_{\nu}(\rho_{n+}) \leq - \hf Mn^{-1/2}$.
	Since $\lambda_\nu(\rho)$ is continuous in $\rho$, 
	there exists $\rho_{\nu *} \in (\rho_{\nu -},\rho_{\nu +})$ such that $\lambda_\nu(\rho_{\nu *}) = 0$; note from the Schur complement decomposition
	\begin{equation*} \label{eq:Schur_decomp}
	\det \left( R- \rho I_{m+p} \right) = \det \left( R_{22} - \rho I_p \right)  \, \det \left( K( \rho )- \rho I_m \right) 
	\end{equation*}
    that $\rho_{\nu *}$ is an eigenvalue of $R$. This is almost surely $\lhat_\nu$, since $\lhat_\nu, \rho_{\nu n} \toas \rho_\nu$,  and $\rho_\nu = \rho(\ell_\nu, \gamma)$ is different from the almost sure limit of any eigenvalue of $R$ adjacent to $\lhat_\nu$ (given by \eqref{eq:CovLimits}), because $\ell_\nu$ is simple and supercritical.
	Therefore, we have $\lhat_\nu \in (\rho_{n-}, \rho_{n+} )$, and thus $|\lhat_\nu- \rho_{\nu n}| \leq Mn^{-1/2}$, 
	which proves \eqref{eq:evaltight}.

}

\newpage

\bibliographystyle{myChicago}
\bibliography{refs}

\end{document}